\documentclass[10pt]{amsart}

\usepackage{amsfonts,amssymb,stmaryrd,amscd,amsmath,latexsym,amsbsy,bbold}

\usepackage{amssymb}
\usepackage{amsfonts}
\usepackage{latexsym}

\newtheorem{theorem}{Theorem}[section]
\newtheorem{lemma}[theorem]{Lemma}
\newtheorem{proposition}[theorem]{Proposition}
\newtheorem{corollary}[theorem]{Corollary}
\theoremstyle{definition}
\newtheorem{definition}[theorem]{Definition}

\newtheorem{example}[theorem]{Example}

\newtheorem{remark}[theorem]{Remark}

\newcommand{\End}{\text{End}}
\newcommand{\Hom}{\text{Hom}}
\newcommand{\Ind}{\text{Ind}}
\newcommand{\Res}{\text{Res}}
\newcommand{\diag}{\text{diag}}
\newcommand{\Span}{\text{Span}}
\newcommand{\sgn}{\text{sgn}}
\newcommand{\tr}{\text{tr}\,}

\newcommand{\kk}{\mathfrak{k}}

\newcommand{\kp}{\mathfrak{p}}
\newcommand{\ka}{\mathfrak{a}}
\newcommand{\kg}{\mathfrak{g}}

\def\HH{\hbox{${\mathcal H}$\kern-5.2pt${\mathcal H}$}}

\newcommand{\cellsize}{22}
\newlength{\cellsz} 
\newsavebox{\cell}
\sbox{\cell}{\begin{picture}(\cellsize,\cellsize)
\put(0,0){\line(1,0){\cellsize}}
\put(0,0){\line(0,1){\cellsize}}
\put(\cellsize,0){\line(0,1){\cellsize}}
\put(0,\cellsize){\line(1,0){\cellsize}}
\end{picture}}
\newcommand\cellify[1]{\def\thearg{#1}\def\nothing{}%
\ifx\thearg\nothing
\vrule width0pt height\cellsz depth0pt\else
\hbox to 0pt{\usebox{\cell} \hss}\fi%
\vbox to \cellsz{
\vss
\hbox to \cellsz{\hss$#1$\hss}
\vss}}
\newcommand\tableau[1]{\vtop{\let\\\cr
\baselineskip -16000pt \lineskiplimit 16000pt \lineskip 0pt
\ialign{&\cellify{##}\cr#1\crcr}}}

\begin{document}
\title[Representations of degenerate affine Hecke algebras of type $BC_{n}$]{On some representations of degenerate affine Hecke algebras of type $BC_{n}$}

\author[Xiaoguang Ma]{Xiaoguang Ma}

\address{Department of Mathematics,
Massachusetts Institute of Technology,
Cambridge, MA  02139, USA}
\email{xma@math.mit.edu}


\maketitle

\section{Introduction}
The degenerate affine Hecke algebra (dAHA) of any finite Coxeter group
was defined by Drinfeld and Lusztig(\cite{Dri},\cite{Lus}). It is
generated by the group algebra of the Coxeter group and by the
commuting generators $y_{i}$ with some relations.

In \cite{AS}, the authors give a Lie-theoretic construction of
representations of the dAHA of type $A_{n-1}$. They construct a functor from the BGG category of $\mathfrak{sl}_N$ to the category of finite dimensional representations of the dAHA of type $A_{n-1}$. They also describe the image of some modules, e.g. the Verma modules, under this functor. In \cite{CEE}, this construction is generalized from the BGG category to the category of $\mathfrak{sl}_N$-bimodules and is upgraded to a Lie-theoretic construction of
representations of degenerate double affine Hecke algebra (dDAHA) of type $A_{n-1}$. 

In \cite{EFM}, the authors generalize the Lie-theoretic constructions in \cite{AS} and \cite{CEE} to the type $BC_{n}$ root system. They construct a functor: 
\begin{eqnarray*}
F_{n,p,\mu}:&\text{ 
$\mathfrak{gl}_{N}$-modules} &\to \begin{array}{c}\text{finite dimensional representations of dAHA of } \\\text{type $BC_{n}$ with special parameters,}\end{array}\\
&M&\mapsto (M\otimes (\mathbb{C}^{N})^{\otimes n})^{\mathfrak{k}^{0},\mu},
\end{eqnarray*}
where $\mathfrak{k}^{0}$ is the subalgebra of $\mathfrak{gl}_{p}\times \mathfrak{gl}_{q}$ generated by trace zero matrices and $\mu\in \mathbb{C}$. They also upgrade this construction to the dDAHA of type $BC_{n}$.
 
In this paper, we first compute the dimension of the image of the principle series module $H_{\pi\otimes\nu}$ under the functor $F_{n,p,\mu}$. Then 
we construct a family of vectors such that they are common eigenvectors 
of the commutative generators $\{y_{i}\}$ of the type $BC_{n}$ dAHA. 
Then we prove the main result (Theorem \ref{main}) which descirbes
the dAHA module structure of $F_{n,p,\mu}(H_{\pi\otimes\nu})$. In the rank 1 case, we also write down an explicit formula for the generator $y_{1}$ in terms of central elements of $U(\mathfrak{g})$. We calculate the eigenvalue of $y_{1}$ by using the central characters.

The paper is organized as follows.  In Section 2, we recall some basic knowledge about principal series modules, representations of unitary groups, representations of symmetric groups and the Schur-Weyl duality. In Section 3, 
we recall some results in \cite{AS} and \cite{EFM}. In Section 4, we compute the dimension of the image of the functor $F_{n,p,\mu}$ for a principal series module of $U(p,q)$, construct a family of common 
eigenvectors for $\{y_{i}\}$ and prove the main result of the paper. 
In Section 5, we consider the case when the rank of type $BC_{n}$ dAHA is $1$. We use the central character of the principal series modules to compute the eigenvalue of $y_{1}$.

\section{Preliminaries}

\subsection{Principal series modules for real reductive Lie groups}

Let $G$ be a reductive Lie group and $K$ be its maximal compact subgroup. 
Let $\mathfrak{g}_{\mathbb{R}}$ and $\mathfrak{k}_{\mathbb{R}}$ be the real Lie algebras of $G$ and $K$, respectively. For $\mathfrak{g}_{\mathbb{R}}$, we have the Cartan decomposition 
$\mathfrak{g}_{\mathbb{R}}=\mathfrak{k}_{\mathbb{R}}\oplus\kp$.
Let $\ka\subset\kp$
be the maximal abelian algebra in $\kp$. Let $A$ be the corresponding connected Lie group. Denote by $M$ the centralizer of $A$ in $K$.
Under the adjoint action of $\ka$ on $\mathfrak{g}_{\mathbb{R}}$, we have a weight decomposition
$\mathfrak{g}_{\mathbb{R}} =\bigoplus_{\lambda}\kg(\lambda)$. The weight $\lambda$ such that $\mathfrak{g}(\lambda)\neq 0$ is called a restricted root of $\mathfrak{a}$ in $\mathfrak{g}$. Let $\Delta^{\mathrm{res}}$ be the set of all restricted roots.
Let $\mathfrak{n}_{+}=\oplus_{\lambda\in \Delta^{\mathrm{res}}_{+}}\mathfrak{g}(\lambda)$ and $N$ be the corresponding Lie group. We have the Iwasawa decomposition $G=K\times A\times N$.

For any element $\nu \in \ka^{*}_{\mathbb{C}}$, we define 
a character of $A$ by 
$$\nu (\exp X)=\exp \nu (X), \text{ for }X\in \ka.$$
We will denote $\nu(a)$ by $a^{\nu}$ for any $a\in A$.
Let $(\pi,W)$ be a finite dimensional irreducible representation of $M$. 
Define the space 
$$H_{\pi\otimes \nu}=\{f: G\to W|f \text{ is measurable}, f|_{K} \text{ is square integrable,}$$
$$\qquad\text{and for all}
\ m\in M, a\in A, n\in N, g\in G, 
f(gman)=a^{-(\nu+\rho)}\pi(m^{-1})f(g)\}.$$

We define the representation $\iota_{\pi\otimes \nu}$ of $G$ on 
$H_{\pi\otimes \nu}$ by 
$$\iota_{\pi\otimes\nu}(g)f(x)=f(g^{-1}x), g,x\in G.$$
This representation is called the principal series representation of $G$ with parameters $\pi$ and $\nu$. It is easy to see that:
$$\iota_{\pi\otimes\nu}=\Ind^{G}_{MAN}(\pi\otimes \nu\otimes 1).$$  

The principal series representations have the following nice properties which are used later.
\begin{proposition}[See \cite{Vog}]{\label{psmod}}
\begin{enumerate}
\item The principal series representation is an admissible representation of $G$. Every irreducible admissible representation of $G$ is infinitesimally equivalent to a composition factor of a principal series representation of $G$. 
\item The restriction of the principal series module to $K$ is the induced representation from the representation $\pi$ of $M$ to $K$, i.e. we have
$$\Res^{G}_{K}H_{\pi\otimes \nu}=\Ind^{K}_{M}W.$$ In particular, $\Res^{G}_{K}H_{\pi\otimes \nu}$ does not depend on $\nu$.
\end{enumerate}
\end{proposition}

\subsection{Principal series modules for $U(p,q)$}\label{U(p,q)}
From now on, we will use the following notations.
Let $p,q$ be two positive integers and $N=p+q$. Without loss of generality, we suppose $q\geq p$.
Let $G=U(p,q)$ and $K=U(p)\times U(q)$ be its maximal compact subgroup. Let $\mathfrak{g}_{\mathbb{R}}$ be the real Lie algebra of $G$,
and  $\mathfrak{g}=\mathfrak{gl}_{N}$. Let $\mathfrak{k}_{\mathbb{R}}$ be the real Lie algebra of $K$,
and  $\mathfrak{k}=\mathfrak{gl}_{p}\times\mathfrak{gl}_{q}$.

The Cartan decomposition of $\mathfrak{g}_{\mathbb{R}}$ is given by $\mathfrak{g}_{\mathbb{R}}=\mathfrak{k}_{\mathbb{R}}\oplus\kp$, where 
$$\kp=\{M\in \text{Mat}_{N\times N}(\mathbb{C})|M=\left(\begin{array}{cc}0 & B \\\bar{B}^{t} & 0\end{array}\right), \text{ where } B\in \text{Mat}_{p\times q}(\mathbb{C}) \}.$$

Then by direct computation, we can find that 
\begin{equation*}
\mathfrak{a}=\mathbb{R}^{p}
=\{\left(\begin{array}{c|cc}0 & D & 0 \\
\hline D & 0 & 0 \\0 & 0 & 0\end{array}\right)
|D=\diag(a_{1},\ldots,a_{p}), a_{i}\in \mathbb{R}\}.
\end{equation*}

Let $\tilde{M}=\underbrace{U(1)\times U(1)\times\ldots\times U(1)}_{2p \text{ times}}\times U(q-p)$.
The centralizer of $A$ in $K$ is a subgroup of $\tilde{M}$:
\begin{equation*}\label{def-M}
M=\Delta(\underbrace{U(1)\times U(1)\times\ldots\times U(1)}_{p \text{ times}})\times U(q-p)\cong \underbrace{U(1)\times U(1)\times\ldots\times U(1)}_{p \text{ times}}\times U(q-p),
\end{equation*}
where $\Delta: \underbrace{U(1)\times U(1)\times\ldots\times U(1)}_{p \text{ times}}\to \underbrace{U(1)\times U(1)\times\ldots\times U(1)}_{2p \text{ times}}$ is the diagonal embedding.
Let $A$ be the Lie group corresponding to $\mathfrak{a}$.
The restricted roots and the basis for the root spaces are given in table 1.

{\tiny
\begin{table}[h]

\caption{ Restricted roots and restricted root spaces for $U(p,q)$}
\tabcolsep 3mm
\renewcommand{\arraystretch}{1.5}
\begin{tabular}{|r|p{7cm}|} 
\hline 
Restricted roots & Basis for the restricted root spaces $\kg_{\lambda}$ \\ 
\cline{1-2}
$\lambda=a_{i}-a_{j}$   &
$v^{1}_{ij}=(E_{i,j}-E_{j,i})+(E_{p+i,p+j}-E_{p+j,p+i})
+(E_{i,p+j}+E_{p+j,i})+(E_{j,p+i}+E_{p+i,j}),$\\ 
$(1\leq i<j\leq p)$ &
$w^{1}_{ij}=\sqrt{-1}(E_{i,j}+E_{j,i})+\sqrt{-1}(E_{p+i,p+j}+E_{p+j,p+i})
+\sqrt{-1}(E_{i,p+j}-E_{p+j,i})-\sqrt{-1}(E_{j,p+i}-E_{p+i,j})$.\\
\cline{1-2}
$\lambda=a_{j}-a_{i}$   &
$v^{2}_{ij}=(E_{i,j}-E_{j,i})+(E_{p+i,p+j}-E_{p+j,p+i})
-(E_{i,p+j}+E_{p+j,i})-(E_{j,p+i}+E_{p+i,j}),$\\ 
$(1\leq i<j\leq p)$ &
$w^{2}_{ij}=\sqrt{-1}(E_{i,j}+E_{j,i})+\sqrt{-1}(E_{p+i,p+j}+E_{p+j,p+i})
-\sqrt{-1}(E_{i,p+j}-E_{p+j,i})+\sqrt{-1}(E_{j,p+i}-E_{p+i,j})$.\\
\cline{1-2}
$\lambda=a_{i}+a_{j}$   &
$v^{3}_{ij}=(E_{i,j}-E_{j,i})-(E_{p+i,p+j}-E_{p+j,p+i})
-(E_{i,p+j}+E_{p+j,i})+(E_{j,p+i}+E_{p+i,j}),$\\ 
$(1\leq i<j\leq p)$ &
$w^{3}_{ij}=\sqrt{-1}(E_{i,j}+E_{j,i})-\sqrt{-1}(E_{p+i,p+j}+E_{p+j,p+i})
-\sqrt{-1}(E_{i,p+j}-E_{p+j,i})-\sqrt{-1}(E_{j,p+i}-E_{p+i,j})$.\\
\cline{1-2}
$\lambda=-a_{i}-a_{j}$   &
$v^{4}_{ij}=(E_{i,j}-E_{j,i})-(E_{p+i,p+j}-E_{p+j,p+i})
+(E_{i,p+j}+E_{p+j,i})-(E_{j,p+i}+E_{p+i,j}),$\\ 
$(1\leq i<j\leq p)$ &
$w^{4}_{ij}=\sqrt{-1}(E_{i,j}+E_{j,i})-\sqrt{-1}(E_{p+i,p+j}+E_{p+j,p+i})
+\sqrt{-1}(E_{i,p+j}-E_{p+j,i})+\sqrt{-1}(E_{j,p+i}-E_{p+i,j})$.\\
\cline{1-2}
$\lambda=a_{i}$   &
$v^{5}_{ij}=(E_{p+i,p+j}-E_{p+j,p+i})+(E_{i,p+j}+E_{p+j,i}),$\\
$(1\leq i\leq p, j>p)$ &
$w^{5}_{ij}=\sqrt{-1}(E_{p+i,p+j}+E_{p+j,p+i})+\sqrt{-1}(E_{i,p+j}-E_{p+j,i})$.\\
\cline{1-2}
$\lambda=-a_{i}$   &
$v^{6}_{ij}=(E_{p+i,p+j}-E_{p+j,p+i})-(E_{i,p+j}+E_{p+j,i}),$\\
$(1\leq i\leq p, j>p)$ &
$w^{6}_{ij}=\sqrt{-1}(E_{p+i,p+j}+E_{p+j,p+i})-\sqrt{-1}(E_{i,p+j}-E_{p+j,i})$.\\
\cline{1-2}
$\lambda=2a_{i},(1\leq i\leq p)$   &
$v^{7}_{i}=\sqrt{-1}(E_{p+i,p+i}-E_{i,i})+\sqrt{-1}(E_{i,p+i}-E_{p+i,i})$.\\
\cline{1-2}
$\lambda=-2a_{i},(1\leq i\leq p)$   &
$w^{7}_{i}=-\sqrt{-1}(E_{p+i,p+i}-E_{i,i})+\sqrt{-1}(E_{i,p+i}-E_{p+i,i})$.\\
\cline{1-2}
\end{tabular} 

\end{table}
}
\begin{remark}{\label{posroot}}
In the future, we will choose $\Delta_{+}^{\mathrm{res}}=\{a_{j}-a_{i},\  -a_{i}-a_{j},\ -a_{i}, \ -2a_{i}\}$ as the set of positive restricted roots when $p\neq q$. When $p=q$, we will choose 
$\Delta_{+}^{\mathrm{res}}=\{a_{j}-a_{i},\  -a_{i}-a_{j}, \ -2a_{i}\}$ as the set of positive restricted roots. We will denote $\mathfrak{n_{+}}=\oplus_{\lambda\in \Delta^{\mathrm{res}}_{+}}\mathfrak{g}(\lambda)$.
\end{remark}

\subsection{Representation theory for unitary groups}
The representation theory of the unitary group $U(n)$ is the same as the
representation theory of $GL_{n}(\mathbb{C})$. All the finite dimensional irreducible unitary representations of $U(n)$ is classified by the dominant weight, 
i.e. by a sequence of integers $\xi=(\xi_{1},\ldots, \xi_{n})$ \linebreak with 
$\xi_{1}\geq \xi_{2}\geq \cdots \geq \xi_{n}$. Denote an irreducible module with weight $\xi$ by $V(\xi)$.

We can see the following conclusions immediately.
\begin{enumerate}
\item[i)] The finite dimensional irreducible representations of $\tilde{M}$ have the form:
$$\bigotimes_{i=1}^{p}V(\alpha_{i})\otimes \bigotimes_{i=1}^{p}V(\beta_{i})\otimes V(\xi)$$ where $\alpha_{i}$ and $\beta_{i}$ are integers and $\xi=(\xi_{1},\ldots,\xi_{q-p})$ is a dominant weight. As a representation of $M$, it is 
$$\bigotimes_{i=1}^{p}V(\alpha_{i}+\beta_{i})\otimes V(\xi),$$
which is irreducible, and all finite dimensional irreducible representations of $M$ have this form.

\item[ii)] As the vector representation of $U(N)$, $\mathbb{C}^{N}=V(1,0,\ldots,0)$.
Denote it by $V_{N}$ and denote the trivial module by 
$\mathbb{1}$.
If we restrict the 
vector representation to $\tilde{M}$ in the natural way, we have a decomposition :
\begin{equation}{\label{decomp-C}}
\mathbb{C}^{N}=V^{1}_{1}\oplus\cdots\oplus V^{2p}_{1}\oplus V_{q-p},
\end{equation}
where $V^{i}_{1}=\mathbb{1}\otimes\cdots\otimes \mathbb{1}\otimes V_{1}\otimes \mathbb{1}\otimes\cdots\otimes\mathbb{1}$ is the irreducible representation of $\tilde{M}$ with $V_{1}$ appearing on the $i$-th component of the tensor product.
\end{enumerate}

\subsection{Irreducible representations of symmetric groups}

Let $\mathcal{S}_{m}$ be the symmetric group. It is well known that all 
its irreducible representations are in $1-1$ correspondence to the partitions 
of the integer $m$.

In fact, for each partition $\lambda=(\lambda_{1},\ldots,\lambda_{s})$ of $m$, we have an irreducible representation of $\mathcal{S}_{m}$, denoted by $S^{\lambda}$, which is called the Specht module. Any finite dimensional irreducible representation of $\mathcal{S}_{m}$ is isomorphic to such a module and we have 
$$d_{\lambda}:=\dim S^{\lambda}=\frac{m!}{\Pi_{k=1}^{|\lambda|} h_{k}(\lambda)},$$
where $h_{k}(\lambda)$ is the hook length of the Young diagram corresponding to $\lambda$, and $|\lambda|=\sum_{i=1}^{s}\lambda_{i}$.
For more details, see \cite{Ful}.

Let $S^{\lambda}$ be the Specht module corresponding to the
partition $\lambda$.
Consider the Jucys-Murphy elements 
\begin{equation*}
L_{s}=\sum_{j<s}S_{sj}, \text{ for }s=2,\ldots,m, \text{ and }L_{1}=0. 
\end{equation*}
These elements commute, and Murphy \cite{Mur} constructed a basis of $S^{\lambda}$ consisting of common eigenvectors for $L_{s}$. Now let us recall the construction of such basis.

Let $\{T_{i}\}_{i=1}^{d_{\lambda}}$ be set of the standard Young tableaux with shape $\lambda$. 
For an element at position $(k,l)$ in a standard Young tableau, its class
is defined to be $l-k$.
Let $\alpha_{i,s}$ for $i=1,\ldots,n$ be the class of the position where $i$ sits in $T_{s}$. 
For example, if we consider the standard Young tableau $T_{1}$:
\begin{equation*}
T_{1}=\tableau{1&2&3&\cdots&n_{1}\\
{\scriptstyle n_{1}+1}& {\scriptstyle n_{1}+2}&\cdots&{\scriptstyle  n_{1}+n_{2}}\\\cdots&\cdots&\cdots\\
\cdots&m} 
\end{equation*}
and put the class in the boxes, we get:
\begin{equation*}
\tableau{0&1&2&\cdots&\cdots&{\scriptstyle n_{1}-1}\\
-1&0&1&\cdots&{\scriptstyle n_{2}-2}\\-2&-1&\cdots\\
\cdots&{\scriptstyle m-s}} 
\end{equation*}

In fact, in this case, we have:
$$\alpha_{i,1}=i-(\lambda_{1}+\cdots+\lambda_{k-1})-k,
\text{ for } \lambda_{1}+\cdots+\lambda_{k-1}<i\leq \lambda_{1}+\cdots+\lambda_{k} .$$

Let $e_{s}=e_{T_{s}}=\sum_{\sigma \in C_{T_{s}}}\sgn(\sigma)\sigma \{T_{s}\}$ be the standard basis for $S^{\lambda}$, where $C_{T_{s}}$ is the column permutations of $T_{s}$ which is a subgroup of $\mathcal{S}_{m}$, $\sgn$ is the sign function of the permutation $\sigma$, and $\{T_{s}\}$ is the element in $S^{\lambda}$ corresponding to the standard Young tableau $T_{s}$. Define
\begin{equation*}
E_{s}=\prod_{c=-m+1}^{m-1}\prod_{\{i|\alpha_{i,s}\neq c,i\leq m\}}
\frac{c-L_{i}}{c-\alpha_{i,s}}.
\end{equation*}
Then we know from \cite{Mur} that $\{w_{s}|w_{s}=E_{s}e_{s}\}$ is a basis for $S^{\lambda}$, and $L_{i}w_{s}=\alpha_{i,s}w_{s}$.

Now define a new family of operators as follows:
\begin{equation*}
\hat{L}_{i}=\sum_{i<j}S_{ij}, \text{ for } i=1,\ldots, m-1, \text{ and } \hat{L}_{m}=0.
\end{equation*}

We have the following lemma:
\begin{lemma}\label{eigenvector}
Let $\sigma=\prod_{i=1}^{[\frac{m}{2}]}S_{i(m-i+1)}$. Then 
$\hat{w}_{s}=\sigma E_{s}e_{s}$ are common eigenvectors of $\hat{L}_{i}$ for $i=1,\ldots,m$, $s=1,\ldots, d_{\lambda}$. The eigenvalues are 
\begin{equation*}
\hat{\alpha}_{i,s}=\alpha_{m-i+1,s} \text{\ for } i\neq m, \text{ and }
\hat{\alpha}_{m,s}=0. 
\end{equation*}
\end{lemma}
\begin{proof}
Notice that
$\hat{L}_{i}=\sigma L_{m-i+1} \sigma$. 
So $$\hat{L}_{i}\hat{w}_{s}=\sigma L_{m-i+1}E_{s}e_{s}=\alpha_{m-i+1,s}\sigma E_{s}e_{s}=\hat{\alpha}_{i,s}\hat{w}_{s}.$$ 
\end{proof}

\subsection{Schur-Weyl duality}\label{Schur-Weyl}

Let $V=\mathbb{C}^{N}$ be the vector representation of $G=GL(N,\mathbb{C})$. Let $\mathcal{S}_{m}$ act on $V^{\otimes m}$ by permutation and $G$ act on it diagonally.

Then from the duality principle, we have that as an $(\mathcal{S}_{m}\times G)$-module,
$$V^{\otimes m}=\bigoplus_{|\lambda|=m, \text{ height of }\lambda\leq  N}V(\lambda)\otimes S^{\lambda},$$
where $S^{\lambda}$ is the Specht module of $\mathcal{S}_{m}$ corresponding to $\lambda$ and $V(\lambda)$ is the highest weight representation of $G$
corresponding to $\lambda$.

\section{The Lie-theoretic construction of representations of the degenerate affine Hecke algebras}\label{knownresults}
\subsection{Degenerate affine Hecke algebras}
In this section, let us recall the definition of degenerate affine Hecke algebras. For more details, see \cite{Ch}.

Let $\mathfrak{h}$ be a finite dimensional real vector space with a
positive definite symmetric bilinear form $(\cdot,\cdot)$.  Let
$\{\epsilon_{i}\}$ be a basis for $\mathfrak{h}$ such that
$(\epsilon_{i},\epsilon_{j})=\delta_{ij}$.  Let $R$ be an irreducible root
system in $\mathfrak{h}$ (possibly non-reduced). 
Let $R_{+}$ be the set of positive roots of $R$, and
let $\Pi=\{\alpha_{i}\}$ be the set of simple roots.
For any root $\alpha$, the corresponding 
coroot is $\alpha^{\vee}=2\alpha/(\alpha,\alpha)$. 
Let $P=\Hom_{\mathbb{Z}}(Q^{\vee},\mathbb{Z})$ be the weight lattice. 
 
Let $\mathcal{W}$ be the Weyl group of $R$ which is generated by $\Sigma$, the set of reflections in $\mathcal{W}$. Let 
 $S_\alpha\in \Sigma$ be the reflection corresponding to the root
$\alpha$. In particular, we write $S_i$ for the simple
reflections $S_{\alpha_i}$.

\begin{definition}
For any conjugation invariant
function $\kappa: \Sigma\to \mathbb{C}$, the {\em degenerate affine Hecke algebra (dAHA)}
$\mathcal{H}(\kappa)$ is the quotient of the free product 
$\mathbb{C}\mathcal{W} * S\mathfrak{h}$ by the 
relations
$$
S_i y-y^{S_i}S_i=\kappa(S_i)\alpha_i(y), y\in \mathfrak{h}.
$$
\end{definition}

Here are two examples.
\begin{example}[dAHA of type $A_{n-1}$]
In this case, $\kappa$ reduces to a complex parameter. The dAHA of type $A_{n-1}$, $\mathcal{H}_{n}(\kappa)$, is generated by 
$\mathbb{C}\mathcal{S}_{n}$, and $y_{1}, \ldots, y_{n}$ with relations:
\begin{eqnarray*}
\sum_{i=1}^{n}y_{i}=0;\quad [S_{i},y_{j}]=0,\quad \forall j\neq i,i+1;\quad [y_{i},y_{j}]=0;
\quad S_{i}y_{i}-y_{i+1}S_{i}=\kappa.
\end{eqnarray*}
For any $c\neq 0$, we have an isomorphism $\mathcal{H}_{n}(\kappa)\cong \mathcal{H}_{n}(c\kappa)$. 
 
\end{example}
\begin{example}[dAHA of type $BC_{n}$]
The function $\kappa$
reduces to two parameters $\kappa=(\kappa_1,\kappa_2)$.

Let $\mathcal{W}_{BC_{n}}=\mathcal{S}_n\ltimes (\mathbb{Z}_{2})^n$ 
be the Weyl group of type $BC_{n}$.
We denote by $S_{ij}$ the reflection in this group 
corresponding to the root $\epsilon_i-\epsilon_j$, 
and by $\gamma_{i}$ the reflection corresponding to
$\epsilon_{i}$. Then $\mathcal{W}$ is generated by 
$S_{i}=S_{i(i+1)}, i=1,\ldots, n-1$, and $\gamma_{n}$. 

The type $BC_{n}$ 
dAHA $\mathcal{H}_{n}(\kappa_1,\kappa_2)$
is generated by $y_{1},\ldots,y_{n}$ 
and $\mathbb{C}\mathcal{W}_{BC_{n}}$ with relations: 
\begin{eqnarray*}
S_{i}y_{i}-y_{i+1}S_{i}=\kappa_{1};\quad
[S_{i},y_{j}]=0,\quad \forall j\neq i,i+1;\\
\gamma_{n}y_{n}+y_{n}\gamma_{n}=\kappa_2;\quad 
[\gamma_{n},y_{j}]=0,\quad \forall j\neq n; \quad [y_{i},y_{j}]=0. 
\end{eqnarray*}
For any $c\neq 0$, we have an isomorphism $\mathcal{H}_{n}(\kappa_1,\kappa_2)\cong \mathcal{H}_{n}(c\kappa_1,c\kappa_2)$. 
\end{example}

\subsection{Type $A_{n-1}$ case}

Consider the Harish-Chandra pair 
$(\mathfrak{sl}_{N}\times\mathfrak{sl}_{N}, \mathfrak{sl}_{N})$, 
where $\mathfrak{sl}_{N}\subset \mathfrak{sl}_{N}\times\mathfrak{sl}_{N}$ is the diagonal. Let $\mathfrak{g}=\mathfrak{sl}_{N}$ and let $HC(\mathfrak{g})$ be the category of
Harish-Chandra bimodules, i.e. the category of $\mathfrak{g}$-bimodules which are locally finite under the adjoint action of $\mathfrak{g}$.

For any $M\in HC(\mathfrak{g})$, we define 
$$F_{n}(M)=(M\otimes (\mathbb{C}^{N})^{\otimes n})^{\mathfrak{g}},$$
where $\mathfrak{g}$ acts on $M$ by the adjoint action and acts on $\mathbb{C}^{N}$ by the vector representation. Let the $A_{n-1}$-type Weyl group $\mathcal{S}_{n}$ act on $F_{n}(M)$
by permutation of the components of $(\mathbb{C}^{N})^{\otimes n}$.

Let $\{b_{k}\}$ be an othonormal basis for $\mathfrak{g}$ and define
$$y_{i}=-\sum_{k} b_{k}\otimes (b_{k})_{i}-\sum_{1\leq j<i}S_{ij}-\frac{n-1}{2},$$
where $S_{ij}\in \mathcal{S}_{n}$ is the permutation $(i,j)$, and the action of $ b_{k}\otimes (b_{k})_{i}$ on $F_{n}(M)$ is defined by letting the first $b_{k}$ act on $X$ and the second one act on the $i$-th component of $(\mathbb{C}^{N})^{\otimes n}$ by vector representation. Then we have the following results.

\begin{theorem}[\cite{AS}]
$F_{n}$ defines a functor from the category of
Harish-Chandra bimodules to the category of finite dimensional $\mathcal{H}_{n}(-1)$-modules.
\end{theorem}

\begin{remark}
This is not quite the original statement of the paper \cite{AS}. In fact, they constructed a functor from the BGG category for $\mathfrak{gl}(n)$ to the category of  finite dimensional $\mathcal{H}_{n}(-1)$-modules. We can obtain the functor in \cite{AS} by composing the Bernstein-Gelfand functor \cite{BG} (from the BGG category to the category of Harish-Chandra bimodules) with $F_{n}$. Arakawa and Suzuki showed that 
the functor they constructed is exact. They 
also studied the images of the standard modules in the BGG category
under this functor. For more details, see \cite{AS}.
\end{remark}

\subsection{Type $BC_{n}$ case}

Let us recall the construction of the functor $F_{n,p,\mu}$ in \cite{EFM}.
Let $\mathbb{C}^{N}$ be the vector representation of $\mathfrak{g}$.
Let $\mathcal{M}$ be a $\mathfrak{g}$-module. Define 
\begin{equation*}
F_{n,p,\mu}(\mathcal{M})=(\mathcal{M}\otimes (\mathbb{C}^{N})^{\otimes n})^{\mathfrak{k}_{0},\mu},
\end{equation*}
where $\mathfrak{k}_{0}$ is the subalgebra in $\mathfrak{k}=\mathfrak{gl}_{p}\times\mathfrak{gl}_{q}$ consisting of trace zero elements and $(\mathfrak{k}_{0},\mu)$-invariants means for all $x\in \mathfrak{k}_{0}$,
$xv=\mu\chi(x)v$. Here $\chi$ is a character of $\mathfrak{k}$ defined in \cite{EFM}:
$$\chi(\left(\begin{array}{cc}X_1 & 0 \\0 & 0\end{array}\right))=q\tr X_{1}-p\tr X_{2}.$$

The Weyl group $\mathcal{W}_{BC_{n}}$ acts on $F_{n,p,\mu}(\mathcal{M})$
in the following way: the element $S_{ij}$ acts by exchanging 
the $i$-th and $j$-th factors, 
and $\gamma_{i}$ acts by multiplying the $i$-th factor by $J=\left(\begin{array}{cc}I_p &  \\ & -I_q\end{array}\right)$
(here we regard $\mathcal{M}$ as the 0-th factor).

Define elements $\tilde{y}_{k}\in \End(F_{n,p,\mu}(\mathcal{M}))$ as follows: 
\begin{equation}\label{new-y}
\tilde{y}_{k} =-\sum_{i|j}E_{i,j}\otimes(E_{j,i})_{k}, 
\text{ for } k=1,\ldots, n,
\end{equation}
where 
$\sum_{i|j}=\sum_{i=1}^{p}\sum_{j=p+1}^{n}+\sum_{j=1}^{p}\sum_{i=p+1}^{n}$, 
the first component acts 
on $\mathcal{M}$ and the second component acts on the $k$-th factor
of the tensor product. 
\begin{theorem}[\cite{EFM}]\label{aff} 
The above action of $\mathcal{W}$ and 
the elements $\tilde{y}_k$ given by \eqref{new-y} combine into 
a representation of the degenerate affine Hecke algebra 
$\mathcal{H}(\kappa_1,\kappa_2)$ on the space $F_{n,p,\mu}(M)$,
with $\kappa_1=\dfrac{p-q-\mu N}{2},\kappa_2=1$ and
\begin{equation*}
y_{i}=\tilde{y}_{i}+\frac{p-q-\mu N}{2}\gamma_{i}+
\frac{1}{2}\sum_{k>i}S_{ik}-\frac{1}{2}\sum_{k<i}S_{ik}+
\frac{1}{2}\sum_{i\neq k}S_{ik}\gamma_{i}\gamma_{k}.
\end{equation*}
So we have a functor $F_{n,p,\mu}$ 
from the the category of 
$\mathfrak{g}$-modules to the category of 
representations of type $BC_{n}$ dAHA with such parameters.
\end{theorem}

Since we know that the principal series module $H_{\pi\otimes \nu}$ is an admissible $G$-module, the Harish-Chandra part (or the $K$-finite part) $(H_{\pi\otimes \nu})_{HC}$ is a $(\mathfrak{g},K)$-module.
So we can apply the functor $F_{n,p,\nu}$ to $(H_{\pi\otimes \nu})_{HC}$. Since the functor only depends on the $(\mathfrak{k}_{0},\mu)$-invariant part, $F_{n,p,\nu}((H_{\pi\otimes \nu})_{HC})=F_{n,p,\nu}(H_{\pi\otimes \nu})$.

In the next several sections, we will study the module $F_{n,p,\nu}(H_{\pi\otimes \nu})$.

\section{Images of principal series modules under $F_{n,p,\mu}$}
Let $K$, $M$ be the subgroups of $U(p,q)$ defined in section \ref{U(p,q)}.
Now let $(\pi, W)$ be an irreducible unitary representation of $M$ with the form
\begin{equation}\label{W}
W=V(n_{1})\otimes\cdots\otimes  V(n_{p})\otimes
V(\xi),
\end{equation}
where $n_{i}$'s are integers and $\xi=(\xi_{1},\ldots,\xi_{q-p})$ is a dominant weight for $U(q-p)$.

Let $\nu=(\nu_{1},\ldots,\nu_{p})\in \mathfrak{a}^{*}_{\mathbb{C}}$. Notice that $\mathfrak{k}=\mathfrak{k}_{0}+\mathfrak{u}(1)$, and 
for any $X\in \mathfrak{u}(1)\subset \mathfrak{k}$, the action on $F_{n,p,\mu}(H_{\pi\otimes\nu})$ is given by
\begin{equation*}
Xv=(\sum_{i=1}^{p}n_{i}+\sum_{i=1}^{q-p}\xi_{i}+n)v, \quad \forall v\in F_{n,p,\mu}(H_{\pi\otimes\nu}).
\end{equation*}

Let $\tau=\dfrac{1}{N}(\sum_{i=1}^{p}n_{i}+\sum_{i=1}^{q-p}\xi_{i}+n)$ and define $\mathbb{1}_{\vartheta}$ to be a $1$-dimensional representation of $K$ with character $\vartheta=-\mu\chi-\tau$. From the definition of the character $\chi$ of $\kk$ in \cite{EFM}, we have
$$\vartheta(\left(\begin{array}{cc}K_1 & 0 \\0 & K_2\end{array}\right))=-\mu(q\tr K_{1}-p\tr K_{2})-\tau(\tr (K_{1})+\tr(K_{2})).$$
In order to make $\mu\chi$ lifted to a unitary group character for $K$, we assume
$\mu p, \mu q\in \mathbb{Z}$ from now on.

It is easy to see that
\begin{eqnarray*}
F_{n,p,\mu}(H_{\pi\otimes\nu})
=(H_{\pi\otimes\nu}
\otimes(\mathbb{C}^{N})^{\otimes n})^{\mathfrak{k}_{0},\mu}
\cong(H_{\pi\otimes\nu}
\otimes(\mathbb{C}^{N})^{\otimes n}\otimes \mathbb{1}_{\vartheta})^{K}.
\end{eqnarray*}

\begin{remark}
We can omit the condition that  $\mu p, \mu q\in \mathbb{Z}$ if we consider the universal covering of the group $U(p,q)$. For simplicity,
we assume this condition in the future discussion.
\end{remark}

\subsection{Dimension of the image}\label{dimofimage}
By Frobenius reciprocity and proposition \ref{psmod}, we can easily see that 
\begin{eqnarray*}
(H_{\pi\otimes\nu}
\otimes(\mathbb{C}^{N})^{\otimes n}\otimes \mathbb{1}_{\vartheta})^{K}
&\stackrel{\scriptscriptstyle I}{\simeq}&(\Res^{G}_{K}(H_{\pi\otimes\nu}
\otimes(\mathbb{C}^{N})^{\otimes n})\otimes \mathbb{1}_{\vartheta})^{K}\\
\nonumber&\stackrel{\scriptscriptstyle II}{\simeq}&(\Ind^{K}_{M}(W)
\otimes\Res^{G}_{K}(\mathbb{C}^{N})^{\otimes n}\otimes \mathbb{1}_{\vartheta})^{K}\\
\nonumber&\stackrel{\scriptscriptstyle III}{\simeq}&(W
\otimes(\mathbb{C}^{N})^{\otimes n}\otimes \mathbb{1}_{\vartheta})^{M}
\end{eqnarray*}
For any elements in $(H_{\pi\otimes\nu}
\otimes(\mathbb{C}^{N})^{\otimes n}\otimes \mathbb{1}_{\vartheta})^{K}$, it is a linear combination of $f\otimes v$ where $f$ is a $W$-valued function on $G$ and $v\in (\mathbb{C}^{N})^{\otimes n}\otimes \mathbb{1}_{\vartheta}$. The isomorphisms $I, II$ tell us that such an element only depends on $f|_{K}\otimes v$.
The isomorphism $III$ comes from the Frobenius reciprocity. 
Then we have the following:
\begin{lemma}{\label{isospace}}
We have an isomorphism of vector spaces:
\begin{eqnarray*}
(H_{\pi\otimes\nu}\otimes(\mathbb{C}^{N})^{\otimes n})^{\mathfrak{k}_{0},\mu}
&\to& 
(W\otimes (\mathbb{C}^{N})^{\otimes n}\otimes \mathbb{1}_{\vartheta})^{M},\\
\sum f\otimes v  &\mapsto& \sum f|_{K}(e)\otimes v\otimes 1_{\vartheta},
\end{eqnarray*}
where $1_{\vartheta}$ is a nonzero element in $\mathbb{1}_{\vartheta}$.
\end{lemma}

Let 
$$\Psi=\{(\alpha_{1},\beta_{1},\ldots,\alpha_{p},\beta_{p},n_{0})|
\forall i, \alpha_{i}\geq 0, \beta_{i}\geq 0, n_{0}\geq 0, \sum_{i=1}^{p}(\alpha_{i}+\beta_{i})+n_{0}=n\}.
$$

From \eqref{decomp-C}, we have 
\begin{eqnarray*}
\Res^{G}_{M}(\mathbb{C}^{N})^{\otimes n}
=\bigoplus_{\Psi}\bigotimes_{i=1}^{p}V(\alpha_{i}+\beta_{i})
\otimes (V_{q-p})^{\otimes n_{0}},
\end{eqnarray*}
and so we have 
\begin{eqnarray*}
W\otimes \Res^{G}_{M}(\mathbb{C}^{N})^{\otimes n}
=\bigoplus_{\Psi}
\bigotimes_{i=1}^{p} V(n_{i}+\alpha_{i}+\beta_{i})\otimes 
\left(V(\xi)\otimes (V_{q-p})^{\otimes n_{0}}\right).
\end{eqnarray*}
Now let $\mathcal{S}_{n_{0}}$ be the subgroup of $\mathcal{S}_{n}$ generated by permutation of the elements in $(V_{q-p})^{\otimes n_{0}}=(\mathbb{C}^{q-p})^{\otimes n_{0}}\subset(\mathbb{C}^{N})^{\otimes n}$.
From the Schur-Weyl duality, we know that as a $(\mathfrak{gl}(q-p)\times \mathcal{S}_{n_{0}})$-module, we have
\begin{equation*}
V_{q-p}^{\otimes n_{0}}\cong \bigoplus_{\lambda\in P}V(\lambda)\otimes S^{\lambda}, 
\end{equation*}
where $P$ is the set of dominant weights appear in the $V_{q-p}^{\otimes n_{0}}$ as a $\mathfrak{gl}(q-p)$-module.

So we have 
\begin{eqnarray*}\label{decomposition}
&&(W\otimes \Res^{G}_{M}(\mathbb{C}^{N})^{\otimes n}\otimes\mathbb{1_{\vartheta}})^{M}\\
&=&\bigoplus_{\Psi;\lambda\in P}
\left(\bigotimes_{i=1}^{p}V(n_{i}+\alpha_{i}+\beta_{i})\otimes
V(\xi)\otimes V(\lambda)\otimes \mathbb{1}_{\vartheta}\right)^{M}\otimes S^{\lambda}.
\end{eqnarray*}

Since
$$\vartheta|_{M}=(\underbrace{-\mu(q-p)-2\tau,\ldots,-\mu(q-p)-2\tau}_{p \text{ times }}, \underbrace{\mu p-\tau,\ldots,\mu p-\tau}_{q-p \text{ times }} ),$$
the vector space $\bigotimes_{i=1}^{p}V(n_{i}+\alpha_{i}+\beta_{i})\otimes
V(\xi)\otimes V(\lambda)\otimes \mathbb{1}_{\vartheta}$ contains $M$-invariants if and only if 
\begin{enumerate}
\item $V(n_{i}+\alpha_{i}+\beta_{i}-\mu(q-p)-2\tau)$ are trivial modules for $i=1,\ldots, p$;
\item $V(\lambda)$ is the dual representation of $V(\xi_{1}+\mu p-\tau,\ldots,\xi_{q-p}+\mu p-\tau)$.
\end{enumerate}

Thus we have the following theorem.
\begin{theorem}\label{Condition}
Let $W$ be an irreducible representation of $M$ with the form \eqref{W}.
\begin{enumerate}
\item[(i)]$F_{n,p,\mu}(H_{\pi\otimes\nu})=(H_{\pi\otimes\nu}
\otimes(\mathbb{C}^{N})^{\otimes n})^{\mathfrak{k}_{0},\mu}\neq 0$ if and only if the parameters $n_{i}, \xi_{i}$ satisfy the following conditions:  
\begin{enumerate}
\item all $(n_{i}-\mu(q-p)-2\tau)$'s are non-positive integers;
\item $(\xi_{1}+\mu p-\tau,\ldots,\xi_{q-p}+\mu p-\tau)$  is a dominant weight for $U(q-p)$ with $\xi_{1}+\mu p-\tau\leq 0$.
\end{enumerate}

\item[(ii)] For generic parameters $n_{i}, \xi_{j}$ satisfying the conditions in (i), the dimension of the vector space $F_{n,p,\mu}(H_{\pi\otimes\nu})$ is  
\begin{equation*}
\frac{n!\prod_{i=1}^{p} 2^{|n_{i}-\mu(q-p)-2\tau|}}{\prod_{i=1}^{p}|n_{i}-\mu(q-p)-2\tau|!\prod_{k=1}^{|\xi^{\mu}|} h_{k}(\xi^{\mu})},
\end{equation*}
where $h_{k}(\xi^{\mu})$ is the hook length of the Young diagram with shape
$\xi^{\mu}=(-\xi_{q-p}-\mu p+\tau,\ldots, -\xi_{1}-\mu p+\tau)$.
\end{enumerate}
\end{theorem} 

\begin{proof}
It is easy to see (i). Now we prove (ii).
We know that 
$\dim S^{\xi^{\mu}}=d_{\xi}=\dfrac{|\xi^{\mu}|!}{\prod_{k} h_{k}(\xi^{\mu})}$, where $h_{k}$ is the hook length of the Young diagram
with the shape $\xi^{\mu}$ at position $k$.
Define 
$$C_{n_{1},\ldots,n_{p}}^{\mu}=\#\{(\alpha_{1},\ldots, \beta_{p},n_{0})\in \Psi|
\alpha_{i}+\beta_{i}=-n_{i}+\mu(q-p)+2\tau, \forall i=1,\ldots,p\},$$
and we have
$$C_{n_{1},\ldots,n_{p}}^{\mu}=\frac{n!\Pi_{i=1}^{p}2^{\alpha_{i}+\beta_{i}}}{\Pi_{i=1}^{p}(\alpha_{i}+\beta_{i})!n_{0}!}.$$

So
\begin{eqnarray*}
\dim F_{n,p,\mu}(H_{\pi\otimes\nu})
=C_{n_{1},\ldots,n_{p}}^{\mu}\dim S^{\xi^{\mu}}
=\frac{n!\Pi_{i=1}^{p} 2^{ (-n_{i}+\mu(q-p)+2\tau)}}{\Pi_{i=1}^{p} (-n_{i}+\mu(q-p)+2\tau)!\Pi_{k=1}^{|\xi^{\mu}|} h_{k}(\xi^{\mu})}.
\end{eqnarray*}
\end{proof}

From now on, we suppose the parameters $n_{i}, \xi_{j}$ for the $M$-module $W$ are generic and satisfy the conditions in theorem \ref{Condition} (i). 




\subsection{Operator $\tilde{y}_{k}$}
We use the notation in section \ref{U(p,q)}.
From the Iwasawa decomposition, we have that for any element $X\in \mathfrak{g}_{\mathbb{R}}$, 
$X=X_{\mathfrak{k}_{\mathbb{R}}}+X_{\mathfrak{a}}+X_{\mathfrak{n}_{+}}$ which corresponding to the decomposition $\mathfrak{g}_{\mathbb{R}}=\mathfrak{k}_{\mathbb{R}}+\mathfrak{a}+\mathfrak{n}_{+}$ where $\mathfrak{n}_{+}$ is defined in the remark \ref{posroot}.
For any $g\in G$ and $f\in H_{\pi\otimes \nu}$, $g(f)(x)=f(g^{-1}x)$ 
and from lemma \ref{isospace}, the vector $\sum g(f)\otimes v\in F_{n,p,\mu}(H_{\pi\otimes \nu})$ is uniquely determined by \linebreak
$\sum g(f)|_{K}(e)\otimes v=\sum f(g^{-1})\otimes v$.
Thus if we consider the Lie algebra action, we have that
$\sum X(f)\otimes v\in F_{n,p,\mu}(H_{\pi\otimes \nu})$ is uniquely determined by\linebreak 
$\sum (X_{\mathfrak{k}_{\mathbb{R}}}+X_{\mathfrak{a}})(f)|_{K}(e)\otimes v$.

Notice that, for $E_{i,j}\in \kg$, we have:
$$
E_{i,j}=\frac{1}{2}\left((E_{i,j}+E_{j,i})-\sqrt{-1}(\sqrt{-1} E_{i,j}-\sqrt{-1} E_{j,i})\right).
$$
From table 1, we have the following identities:
\begin{eqnarray*}
1\leq i<j\leq p,\quad &&E_{i,p+j}+E_{p+j,i}
=\frac{1}{2}(v^{4}_{i,j}-v^{2}_{i,j})+(E_{p+i,p+j}-E_{p+j,p+i}),\\
&&\sqrt{-1} (E_{i,p+j}- E_{p+j,i})
=\frac{1}{2}(w^{4}_{i,j}-w^{2}_{i,j})+\sqrt{-1}(E_{p+i,p+j}+E_{p+j,p+i});\\
1\leq j<i\leq p, \quad&&
E_{i,p+j}+E_{p+j,i}
=-\frac{1}{2}(v^{4}_{j,i}+v^{2}_{j,i})+(E_{j,i}-E_{i,j}),\\
&&\sqrt{-1}(E_{i,p+j}-E_{p+j,i})
=\frac{1}{2}(w^{4}_{j,i}+w^{2}_{j,i})-\sqrt{-1}(E_{j,i}+E_{i,j});\\
1\leq i\leq p<j, \quad&&
E_{i,p+j}+E_{p+j,i}
= -v^{6}_{i,j}+(E_{p+i,p+j}-E_{p+j,p+i}),\\
&&\sqrt{-1}(E_{i,p+j}-E_{p+j,i})
= -w^{6}_{i,j}+\sqrt{-1}(E_{p+i,p+j}+E_{p+j,p+i});\\
1\leq i=j\leq p,\quad&&
E_{i,p+i}+E_{p+i,i}
= E_{i,p+i}+E_{p+i,i},\\
&&\sqrt{-1}(E_{i,p+i}-E_{p+i,i})
=w_{i}^{7}+\sqrt{-1}(E_{p+i,p+i}-E_{i,i}).
\end{eqnarray*}
Here we choose positive root as in the remark \ref{posroot}. Since $\mathfrak{n}_{+}$ acts on the principal series module by zero, we have
as operators on $H_{\pi\otimes\nu}$,
\begin{equation*}
E_{i,p+j}
=\left\{\begin{array}{cc}
E_{p+i,p+j}, & 1\leq i<j\leq p ;\\
-E_{i,j}, & 1\leq j<i\leq p ;\\
E_{p+i,p+j}, &  1\leq i\leq p<j.
\end{array}\right.
\quad
E_{p+j,i}
=\left\{\begin{array}{cc}
-E_{p+j,p+i}, & 1\leq i<j\leq p ;\\
E_{j,i}, & 1\leq j<i\leq p ;\\
-E_{p+j,p+i}. &  1\leq i\leq p<j.
\end{array}\right.
\end{equation*}
\begin{eqnarray*}
E_{i,p+i}&=&\frac{1}{2}(E_{i,p+i}+E_{p+i,i})+\frac{1}{2}(E_{p+i,p+i}-E_{ii}),\\
E_{p+i,i}&=&\frac{1}{2}(E_{i,p+i}+E_{p+i,i})-\frac{1}{2}(E_{p+i,p+i}-E_{ii}),  1\leq i \leq p.
\end{eqnarray*}

So as operators on $F_{n,p,\mu}(H_{\pi\otimes\nu})$, we have
\begin{eqnarray*}
\tilde{y}_{k}
&=&-\sum_{i|j}E_{i,j}\otimes (E_{j,i})_{k}\\
&=& -\sum_{1\leq i<j\leq p}E_{p+i,p+j}\otimes (E_{p+j,i})_{k}+\sum_{1\leq i<j\leq p}E_{p+j,p+i}\otimes (E_{i,p+j})_{k}\\
&&\quad
+\sum_{1\leq j<i\leq p}E_{i,j}\otimes (E_{p+j,i})_{k}-\sum_{1\leq j<i\leq p}E_{j,i}\otimes (E_{i,p+j})_{k}\\
&&\quad 
-\sum_{1\leq i\leq p<j}E_{p+i,p+j}\otimes (E_{p+j,i})_{k}
+\sum_{1\leq i\leq p<j}E_{p+j,p+i}\otimes (E_{i,p+j})_{k}\\
&&-\sum_{1\leq i\leq p}\frac{1}{2}(E_{i,p+i}+E_{p+i,i})\otimes (E_{p+i,i})_{k}-\sum_{1\leq i\leq p}\frac{1}{2}(E_{p+i,p+i}-E_{i,i})\otimes (E_{p+i,i})_{k}\\
&&
-\sum_{1\leq i\leq p}\frac{1}{2}(E_{i,p+i}+E_{p+i,i})\otimes (E_{i,p+i})_{k}+\sum_{1\leq i\leq p}\frac{1}{2}(E_{p+i,p+i}-E_{i,i})\otimes (E_{i,p+i})_{k}.
\end{eqnarray*}

Then using the invariant property, we have
(as operators on $F_{n,p,\mu}(H_{\pi\otimes\nu})$)
\begin{eqnarray*}
\tilde{y}_{k}
&=& -\sum_{1\leq i<j\leq p}1\otimes (E_{i,p+i})_{k}-
\sum_{1\leq j<i\leq p}1\otimes (E_{p+j,j})_{k}-\sum_{1\leq i\leq p<j}1\otimes (E_{i,p+i})_{k}\\
&&\quad 
-\sum_{1\leq i\leq p}\rho_{i}\otimes (E_{p+i,i})_{k}
-\sum_{1\leq i\leq p}\rho_{i}\otimes (E_{i,p+i})_{k}
-\sum_{1\leq i\leq p}\frac{1}{2}\otimes (E_{p+i,i})_{k}\\
&&\quad
-\sum_{1\leq i\leq p}\frac{1}{2}\otimes (E_{i,p+i})_{k}+\frac{\mu(p+q)}{2}\sum_{1\leq i\leq p}1\otimes (E_{p+i,i})_{k}
-\frac{\mu(p+q)}{2}\sum_{1\leq i\leq p}1\otimes (E_{i,p+i})_{k}
\\
&&\quad
-\sum_{k\neq l}\left(
-\sum_{1\leq i<j\leq p}(E_{p+i,p+j})_{l}\otimes (E_{p+j,i})_{k}+\sum_{1\leq i<j\leq p}(E_{p+j,p+i})_{l}\otimes (E_{i,p+j})_{k}+\right.\\
&&\quad
\sum_{1\leq j<i\leq p}(E_{i,j})_{l}\otimes (E_{p+j,i})_{k}-\sum_{1\leq j<i\leq p}(E_{j,i})_{l}\otimes (E_{i,p+j})_{k}\\
&&\quad 
-\sum_{1\leq i\leq p<j}(E_{p+i,p+j})_{l}\otimes (E_{p+j,i})_{k}
+\sum_{1\leq i\leq p<j}(E_{p+j,p+i})_{l}\otimes (E_{i,p+j})_{k}+\\
&&\quad 
-\sum_{1\leq i\leq p}\frac{1}{2}(E_{p+i,p+i})_{l}\otimes (E_{p+i,i})_{k}
+\sum_{1\leq i\leq p}\frac{1}{2}(E_{i,i})_{l}\otimes (E_{p+i,i})_{k}\\
&&\quad\left.
+\sum_{1\leq i\leq p}\frac{1}{2}(E_{p+i,p+i})_{l}\otimes (E_{i,p+i})_{k}
-\sum_{1\leq i\leq p}\frac{1}{2}(E_{i,i})_{l}\otimes (E_{i,p+i})_{k}\right).
\end{eqnarray*}
Here $\rho_{i}=\dfrac{\nu+\rho}{2}(E_{p+i,i}+E_{i,p+i})$.

When $q>p$, we have
\begin{eqnarray*}
\rho&=&\frac{1}{2}\sum_{\alpha\in \Delta^{\mathrm{res}}_{+}}\alpha
=-\sum_{1\leq i\leq p}(p+q-2i+1)a_{i}.
\end{eqnarray*}
Then
\begin{equation*}
\rho_{i}=-\frac{p+q}{2}+i-\frac{1}{2}+\frac{\nu_{i}}{2}.
\end{equation*}

So as operators on $F_{n,p,\mu}(H_{\pi\otimes\nu})$,
\begin{eqnarray*}
&&\tilde{y}_{k}\\
&=& \sum_{1\leq i\leq p}(\frac{p-q}{2}-\frac{\nu_{i}}{2}-\frac{\mu (p+q)}{2})\otimes (E_{i,p+i})_{k}+\sum_{1\leq i\leq p}(\frac{q-p}{2}-\frac{\nu_{i}}{2}+\frac{\mu(p+q)}{2})\otimes (E_{p+i,i})_{k}\\
&&\quad
-\sum_{k\neq l}\left(
-\sum_{1\leq i<j\leq p}(E_{p+i,p+j})_{l}\otimes (E_{p+j,i})_{k}+\sum_{1\leq i<j\leq p}(E_{p+j,p+i})_{l}\otimes (E_{i,p+j})_{k}\right.\\
&&\quad+\sum_{1\leq j<i\leq p}(E_{i,j})_{l}\otimes (E_{p+j,i})_{k}-\sum_{1\leq j<i\leq p}(E_{j,i})_{l}\otimes (E_{i,j+p})_{k}\\
&&\quad 
-\sum_{1\leq i\leq p<j}(E_{p+i,p+j})_{l}\otimes (E_{p+j,i})_{k}
+\sum_{1\leq i\leq p<j}(E_{p+j,p+i})_{l}\otimes (E_{i,p+j})_{k}+\\
&&\quad 
-\sum_{1\leq i\leq p}\frac{1}{2}(E_{p+i,p+i})_{l}\otimes (E_{p+i,i})_{k}
+\sum_{1\leq i\leq p}\frac{1}{2}(E_{i,i})_{l}\otimes (E_{p+i,i})_{k}\\
&&\quad\left.
+\sum_{1\leq i\leq p}\frac{1}{2}(E_{p+i,p+i})_{l}\otimes (E_{i,p+i})_{k}
-\sum_{1\leq i\leq p}\frac{1}{2}(E_{i,i})_{l}\otimes (E_{i,p+i})_{k}\right).
\end{eqnarray*}

When $q=p$, we have
\begin{equation*}
\rho=\sum_{1\leq i\leq p}(-2p+2i-1)a_{i}.
\end{equation*}
Then 
\begin{equation*}
\rho_{i}=-p+i-\frac{1}{2}+\frac{\nu_{i}}{2}.
\end{equation*}
So as operators on $F_{n,p,\mu}(H_{\pi\otimes\nu})$,
\begin{eqnarray*}
&&\tilde{y}_{k}\\
&=& \sum_{1\leq i\leq p}(-\frac{\nu_{i}}{2}-\mu p)\otimes (E_{i,p+i})_{k}+\sum_{1\leq i\leq p}(-\frac{\nu_{i}}{2}+\mu p)\otimes (E_{p+i,i})_{k}\\
&&\quad
-\sum_{k\neq l}\left(
-\sum_{1\leq i<j\leq p}(E_{p+i,p+j})_{l}\otimes (E_{p+j,i})_{k}+\sum_{1\leq i<j\leq p}(E_{p+j,p+i})_{l}\otimes (E_{i,p+j})_{k}
\right.\\
&&\quad+\sum_{1\leq j<i\leq p}(E_{i,j})_{l}\otimes (E_{p+j,i})_{k}-\sum_{1\leq j<i\leq p}(E_{j,i})_{l}\otimes (E_{i,j+p})_{k}\\
&&\quad 
-\sum_{1\leq i\leq p}\frac{1}{2}(E_{p+i,p+i})_{l}\otimes (E_{p+i,i})_{k}
+\sum_{1\leq i\leq p}\frac{1}{2}(E_{i,i})_{l}\otimes (E_{p+i,i})_{k}
\\
&&\quad 
\left.
+\sum_{1\leq i\leq p}\frac{1}{2}(E_{p+i,p+i})_{l}\otimes (E_{i,p+i})_{k}
-\sum_{1\leq i\leq p}\frac{1}{2}(E_{i,i})_{l}\otimes (E_{i,p+i})_{k}\right).
\end{eqnarray*}

\subsection{Common eigenvectors for $y_{k}$}\label{sect-comm}
Let $W=V(n_{1})\otimes\cdots\otimes V(n_{p})\otimes V(\xi)$ be an irreducible module for $M$
with generic parameters $n_{i},\xi=(\xi_{1},\ldots,\xi_{q-p})$ satisfying the condition in theorem \ref{Condition}. 

Let $n_{i}^{\mu}=-n_{i}+\mu(q-p)+2\tau$, and
$n^{\mu}=(n_{1}^{\mu},\ldots, n^{\mu}_{p})$.
Let $\xi^{\mu}_{i}=-\xi_{q-p-i+1}-\mu p+\tau$, and 
$\xi^{\mu}=(\xi^{\mu}_{1},\ldots,\xi^{\mu}_{q-p})$. 
Let $n^{\mu}_{\xi}=\sum\xi_{i}^{\mu}=-\sum \xi_{i}-\mu p(q-p)+(q-p)\tau$. Notice that all $n_{i}^{\mu}$ and $\xi_{i}^{\mu}$ are nonnegative integers and we have
$\sum_{i=1}^{p}n_{i}^{\mu}+n_{\xi}^{\mu}=n$.
Rewrite $(\mathbb{C}^{N})^{\otimes n}$ as 
$(\mathbb{C}^{N})^{\otimes n_{1}^{\mu}}
\otimes \cdots \otimes (\mathbb{C}^{N})^{\otimes n_{p}^{\mu}}
\otimes  (\mathbb{C}^{N})^{\otimes n_{\xi}^{\mu}}$.
Let $\mathcal{S}_{n_{i}^{\mu}}$ be the subgroup of
$\mathcal{S}_{n}$ which is generated by permutations 
of the tensor factors in $(\mathbb{C}^{N})^{\otimes n_{i}^{\mu}}$.
Let $\mathcal{S}_{n_{\xi}^{\mu}}$ be the subgroup of
$\mathcal{S}_{n}$ which is generated by permutations 
of the tensor factors in
$(\mathbb{C}^{N})^{\otimes n_{\xi}^{\mu}}$.

From section \ref{dimofimage}, we know that any vector in $F_{n,p,\mu}(H_{\pi\otimes\nu})$ has the form:
\begin{equation*}
\sum_{i} \omega^{i}\otimes u_{1}^{i}
\otimes \cdots\otimes u_{p}^{i}\otimes w^{i},
\end{equation*}
where $\omega^{i}\in W\otimes \mathbb{1}_{\vartheta}$,
$u^{i}_{j}\in V(n_{j}^{\mu})\subset 
(\mathbb{C}^{N})^{\otimes n}$, and 
$w^{i}\in V(\xi^{\mu})\otimes S^{\xi^{\mu}}\subset (\mathbb{C}^{N})^{\otimes n}$.

Now choose  a nonzero vector $\omega\in W\otimes \mathbb{1}_{\vartheta}$ with the highest weight, and a nonzero vector $w_{\xi^{\mu}}\in V(\xi^{\mu})\subset 
(\mathbb{C}^{N})^{\otimes n_{\xi}^{\mu}}$
of weight $\xi^{\mu}$. 

For $k=1,\ldots, p$, define
\begin{equation*} 
u_{k}
=\underbrace{(e_{k}- e_{k+p})\otimes\cdots\otimes (e_{k}- e_{k+p})}_{n^{\mu}_{k} \text{ times}}\in (\mathbb{C}^{N})^{\otimes n^{\mu}_{k}},
\end{equation*}
where $e_{i}$ is the vector in $\mathbb{C}^{N}$ with $1$ on the $i$-th position and $0$ elsewhere. It is easy to see that $u_{k}\in V(n_{i}^{\mu})$.

Define 
\begin{equation}{\label{eqn-eigenvector} }
\varpi_{s}=\omega\otimes u_1\otimes\cdots\otimes u_p\otimes (w_{\xi^{\mu}}\otimes \hat{w}_{s}),
\text{ for } s=1,\ldots, d_{\xi^{\mu}},
\end{equation}
where $\hat{w}_{s}$, as in lemma \ref{eigenvector}, are basis for the Specht module $S^{\xi^{\mu}}$ of $\mathcal{S}_{n_{\xi}^{\mu}}$ and are common eigenvectors of $\hat{L}_{i}$ with eigenvalues $\hat{\alpha}_{i,s}$. 

We have the following properties:
\begin{proposition}\label{prop}
\begin{enumerate}
\item[(i)] $\varpi_{s}$ is invariant under the action of $\mathcal{S}_{n^{\mu}_1}\times\cdots\times \mathcal{S}_{n^{\mu}_p}\subset \mathcal{S}_n$.
\item[(ii)]Under the action of $\mathcal{S}_{n^{\mu}_{\xi}}\subset \mathcal{S}_{n}$, $\varpi_{s}$ generates $S^{\xi^{\mu}}$. 
\item[(iii)]The subgroup of $\Gamma$, $1\times\cdots\times 1\times\mathbb{Z}_{2}^{n^{\mu}_{\xi}}$ acts as $-1$. 
\item[(iv)] Under the action of $\mathcal{S}_{n}$, $\{\varpi_{i}\}$ generate $F_{n,p,\mu}(H_{\pi\otimes\nu})$.
\end{enumerate}
\end{proposition}
\begin{proof}
From the construction of $u_{i}$, we can get (i) easily.
Since $\forall \sigma\in \mathcal{S}_{n_{\xi}^{\mu}}$ we have 
$$\sigma(\varpi_{s})=\omega\otimes u_1\otimes\cdots\otimes u_p\otimes w_{\xi^{\mu}}\otimes \sigma(\hat{w}_{s}),$$ we get (ii).

From the construction, $w_{\xi^{\mu}}\in V(\xi^{\mu})\subset (\mathbb{C}^{N})^{\otimes n_{\xi}^{\mu}}$. So $w_{\xi^{\mu}}\in \Span\{e_{i}|i>2p\}$ which implies (iii).
 
We can see that under the action of $\mathcal{W}_{BC_{n}}/S_{n_{\xi}^{\mu}}$, $\varpi_{k}$ generates a vector space with dimension $\dfrac{n!\Pi_{i=1}^{p}2^{n_{i}^{\mu}}}{\Pi_{i=1}^{p}n_{i}^{\mu}!n_{\xi}^{\mu}!}$. Two vector spaces generated by different $\varpi_{i}$ and $\varpi_{j}$ do not intersect since 
the last component $\hat{w}_{i}$ and $\hat{w}_{j}$ are linearly independent.
So the total dimension of the space generated by $\{\varpi_{i}\}$ under the action of $\mathcal{W}_{BC_{n}}$ will be
$$\dfrac{n!\Pi_{i=1}^{p}2^{n_{i}^{\mu}}}{\Pi_{i=1}^{p}n_{i}^{\mu}!n_{\xi}^{\mu}!}\cdot d_{\xi^{\mu}}
=\dfrac{n!\Pi_{i=1}^{p}2^{n_{i}^{\mu}}}{\Pi_{i=1}^{p}n_{i}^{\mu}!\Pi_{i=1}^{n_{\xi}^{\mu}}h_{k}(\xi^{\mu})},$$
which equals to the dimension of $F_{n,p,\mu}(H_{\pi\otimes\nu})$
we find in theorem \ref{Condition}.
This proves (iv).
\end{proof}
\begin{remark}
All the above discussions hold for $q\geq p$. In fact, when $q=p$, $n_{\xi}^{\mu}=0$ and we can still do the above construction of $\varpi_{s}$. The only difference is that we do not have the $V(\xi)\otimes S^{\xi^{\mu}}$ part, which makes things much simpler.  
\end{remark}

For $r=1,\ldots, p$, let $m_{r}=\sum_{i=1}^{r}n^{\mu}_{i}$ and $m_{0}=0, m_{p+1}=n$. Notice that when $q=p$, we have $m_{p}=m_{p+1}=n$.

\begin{theorem}\label{common-eigenvector}
The vector $\varpi_{s}$ defined by \eqref{eqn-eigenvector} is a common eigenvector of the operators $y_{k}$, for $k=1,\ldots,n$, $s=1,\ldots,d_{\xi^{\mu}}$. When $q>p$, the eigenvalue of 
$y_{k}$ on $\varpi_{s}$ is
\begin{equation*}
\lambda_{k,s}=\left\{\begin{array}{cc}\frac{\nu_{r}}{2}+\frac{m_{r}+m_{r-1}}{2}-k+\frac{1}{2}, &\text{for } m_{r-1}<k\leq m_{r}, \text{ where }r\leq p; \\ &  \\-\frac{p-q-\mu(p+q)}{2}+\hat{\alpha}_{k-m_{p}+1,s}, &\text{for } m_{p}<k\leq m_{p+1}. \end{array}\right.
\end{equation*}
When $q=p$, the eigenvalue of $y_{k}$ on $\varpi_{s}$ is 
\begin{equation*}
\lambda_{k,s}=\frac{\nu_{r}}{2}+\frac{m_{r}+m_{r-1}}{2}-k+\frac{1}{2}, \text{ for } m_{r-1}<k\leq m_{r}.
\end{equation*}
\end{theorem}

In the rest of this section, we will prove this theorem. We divide the proof into the following cases corresponding to the range of the index $k$. 
 
\subsubsection{Case 1:~$m_{r-1}<k\leq m_{r}$, for $r\leq p$ and $q>p$.} 

From the construction of $\varpi$, we know that the $k$-th factor of the tensor product must be $e_{r}$ or $e_{r+p}$.
Then as an operator on $\varpi_{s}$, $\tilde{y}_{k}$ can be written as follows:
\begin{eqnarray*}
\tilde{y}_{k}
&=&(\frac{p-q}{2}-\frac{\nu_{r}}{2}-\frac{\mu(p+q)}{2})\otimes(E_{r,p+r})_{k}
+(\frac{q-p}{2}-\frac{\nu_{r}}{2}+\frac{\mu(p+q)}{2})\otimes(E_{p+r,r})_{k}\\
&&\quad 
+\mathop{\sum_{l\neq k}}_{m_{r-1} < l\leq m_{r}}\left(\frac{\gamma_{l}}{2}\otimes(E_{r,p+r})_{k}-\frac{\gamma_{l}}{2}\otimes(E_{p+r,r})_{k}\right)\\
&&\quad
+\sum_{r<t\leq p}\sum_{m_{t-1}<l\leq m_{t}}\left((E_{p+r,p+t})_{l}\otimes (E_{p+t,r})_{k}+(E_{r,t})_{l}\otimes (E_{t,p+r})_{k}
\right)\\
&&\quad
-\sum_{t<r}\sum_{m_{t-1}<l\leq m_{t}}
\left((E_{p+r,p+t})_{l}\otimes (E_{t,p+r})_{k}+(E_{r,t})_{l}\otimes (E_{p+t,r})_{k}\right)\\
&&\quad
+\sum_{k\neq l}\left(\sum_{1\leq i\leq p<j}(E_{p+i,p+j})_{l}\otimes(E_{p+j,i})_{k}
-\sum_{1\leq i\leq p<j}(E_{p+j,p+i})_{l}\otimes(E_{i,p+j})_{k}\right).
\end{eqnarray*}

Notice that the vector $\varpi_{s}$ can be written as
$$\varpi_{s}=\omega\otimes(a\otimes (e_{r})_{k}\otimes b-a\otimes (e_{p+r})_{k}\otimes b)\otimes \hat{w}_{s}.$$
Then we have 
\begin{eqnarray*}
&&\mathop{\sum_{l\neq k}}_{m_{r-1} < l\leq m_{r}}\left(\frac{\gamma_{l}}{2}\otimes(E_{r,p+r})_{k}-\frac{\gamma_{l}}{2}\otimes(E_{p+r,r})_{k}\right)(\varpi_{s})\\
&=&\mathop{\sum_{l\neq k}}_{m_{r-1} < l\leq m_{r}}
\omega\otimes\frac{\gamma_{l}}{2}\left(a\otimes(-e_{r})_{k}\otimes b-a\otimes(e_{p+r})_{k}\otimes b\right)\otimes \hat{w}_{s}\\
&=&-\frac{1}{2}\mathop{\sum_{l\neq k}}_{m_{r-1} < l\leq m_{r}}  S_{lk}\gamma_{l}\gamma_{k}(\varpi_{s}).
\end{eqnarray*}

Similarly, for $m_{t-1}<l\leq m_{t}$, and $t<s$, we can write 
\begin{eqnarray*}
\varpi&=& \omega\otimes (a\otimes (e_{t})_{l}\otimes b\otimes (e_{r})_{k}\otimes c
-a\otimes (e_{t+p})_{l}\otimes b\otimes (e_{r})_{k}\otimes c\\
&&
-a\otimes (e_{t})_{l}\otimes b\otimes (e_{r+p})_{k}\otimes c
+a\otimes (e_{t+p})_{l}\otimes b\otimes (e_{r+p})_{k}\otimes c)\otimes \hat{w}_{s}.
\end{eqnarray*}

Then 
\begin{eqnarray*}
&&-\sum_{t<r}\sum_{m_{t-1}<l\leq m_{t}}
\left((E_{p+r,p+t})_{l}\otimes (E_{t,p+r})_{k}+(E_{r,t})_{l}\otimes (E_{p+t,r})_{k}\right)(\varpi_{s})\\
&=&-\sum_{t<r}\sum_{m_{t-1}<l\leq m_{t}}\omega\otimes(a\otimes (e_{r})_{l}\otimes b\otimes (e_{t+p})_{k}\otimes c\otimes \hat{w}_{s}
+a\otimes (e_{r+p})_{l}\otimes b\otimes (e_{t})_{k}\otimes c\otimes \hat{w}_{s})\\
&=&\sum_{t<r}\sum_{m_{t-1}<l\leq m_{t}}\frac{1}{2}(S_{kl}-S_{kl}\gamma_{k}\gamma_{l})(\varpi_{s}).
\end{eqnarray*}

By a similar method, we can show that
\begin{eqnarray*}
&&\sum_{r<t\leq p}\left(\sum_{m_{t-1}<l\leq m_{t}}
(E_{p+r,p+t})_{l}\otimes (E_{p+t,r})_{k}+(E_{r,t})_{l}\otimes (E_{t,p+r})_{k}\right)(\varpi_{s})\\
&=&-\sum_{r<t\leq p}\sum_{m_{t-1}<l\leq m_{t}}\frac{1}{2}(S_{kl}+S_{kl}\gamma_{k}\gamma_{l})(\varpi_{s}).
\end{eqnarray*}

Now consider the last term.
Since $p+j>2p$ , we have
\begin{eqnarray*}
&&\sum_{k\neq l}\left(\sum_{1\leq i\leq p<j}(E_{p+i,p+j})_{l}\otimes(E_{p+j,i})_{k}
-\sum_{1\leq i\leq p<j}(E_{p+j,p+i})_{l}\otimes(E_{i,p+j})_{k}\right)(\varpi_{s})\\
&=& \left(\sum_{m_{p}<l\leq m_{p+1}}\sum_{p<j}(E_{p+r,p+j})_{l}\otimes(E_{p+j,r})_{k}\right)(\varpi_{s}).
\end{eqnarray*}
Now suppose 
$w_{\xi^{\mu}}\otimes \hat{w}_{s}
=\sum c^{s}_{i_{1},\ldots,i_{n^{\mu}_{\xi}}}e_{i_{1}}
\otimes\cdots \otimes e_{i_{n^{\mu}_{\xi}}}$,
where by our construction, all the indices $i_{u}>2p$.
Then we can write 
\begin{eqnarray*}
\varpi_{s} & = &\omega\otimes a\otimes (e_{r})_{k}\otimes b\otimes (\sum c^{s}_{i_{1},\ldots,i_{n_{\xi}^{\mu}}}e_{i_{1}}\otimes\cdots \otimes e_{i_{n^{\mu}_{\xi}}})\\
&&\quad-\omega\otimes a\otimes (e_{r+p})_{k}\otimes b\otimes (\sum c^{s}_{i_{1},\ldots,i_{n^{\mu}_{\xi}}}e_{i_{1}}\otimes\cdots \otimes e_{i_{n^{\mu}_{\xi}}}).
\end{eqnarray*}

It is easy to see that

$$\sum_{m_{p}<l\leq m_{p+1}}\sum_{p<j}(E_{p+r,p+j})_{l}\otimes(E_{p+j,r})_{k}(\varpi_{s})
=-\sum_{m_{p}<l\leq m_{p+1}}\frac{1}{2}(S_{lk}+S_{lk}\gamma_{l}\gamma_{k})(\varpi_{s}).$$

From above discussion, we have
\begin{eqnarray*}
&&\tilde{y}_{k}(\varpi_{s})\\
&=&\left((\frac{p-q}{2}-\frac{\nu_{r}}{2}- \frac{\mu(p+q)}{2})\otimes(E_{r,p+r})_{k}
+(\frac{q-p}{2}-\frac{\nu_{r}}{2}+\frac{\mu (p+q)}{2})\otimes(E_{p+r,r})_{k}\right.\\
&&\quad 
-\frac{1}{2}\mathop{\sum_{l\neq k}}_{m_{r-1} < l\leq m_{r}}S_{lk}\gamma_{l}\gamma_{k}+
\sum_{t<r}\sum_{m_{t-1}<l\leq m_{t}}\frac{1}{2}(S_{kl}-S_{kl}\gamma_{k}\gamma_{l})\\
&&\quad
\left.-\sum_{r<t\leq p}\sum_{m_{t-1}<l\leq m_{t}}\frac{1}{2}(S_{kl}+S_{kl}\gamma_{k}\gamma_{l})
-\sum_{m_{p}<l\leq m_{p+1}}\frac{1}{2}(S_{lk}+S_{lk}\gamma_{l}\gamma_{k})\right)(\varpi_{s})\\
&=&\left((\frac{p-q}{2}-\frac{\nu_{r}}{2}- \frac{\mu(p+q)}{2})\otimes(E_{r,p+r})_{k}
+(\frac{q-p}{2}-\frac{\nu_{r}}{2}+\frac{\mu (p+q)}{2})\otimes(E_{p+r,r})_{k}\right.\\
&&\quad\left. 
-\frac{1}{2}\sum_{l\neq k}S_{lk}\gamma_{l}\gamma_{k}+
\sum_{t<r}\sum_{m_{t-1}<l\leq m_{t}}\frac{1}{2}S_{kl}
-\sum_{r<t\leq p}\sum_{m_{t-1}<l\leq m_{t}}\frac{1}{2}S_{kl}\right)(\varpi_{s}).\end{eqnarray*}

From the construction of $\varpi_{s}$, we have
\begin{eqnarray*}
\frac{1}{2}\sum_{k<l\leq m_{r}}S_{lk}-\frac{1}{2}\sum_{m_{r-1}<l<k}S_{lk}
=\frac{m_{r}-k}{2}-\frac{k-m_{r-1}-1}{2}
=\frac{m_{r}+m_{r-1}+1}{2}-k.
\end{eqnarray*}

Then
\begin{eqnarray*}
y_{k}(\varpi_{s})&=&\left((\frac{p-q}{2}-\frac{\nu_{r}}{2}-\frac{\mu (p+q)}{2})\otimes(E_{r,p+r})_{k}
+(\frac{q-p}{2}-\frac{\nu_{r}}{2}+\frac{\mu (p+q)}{2})\otimes(E_{p+r,r})_{k}\right.\\
&&\quad\left.+\frac{p-q-\mu(p+q)}{2}\gamma_{k}+\frac{m_{r}+m_{r-1}+1}{2}-k\right)(\varpi_{s}).
\end{eqnarray*}

Now let us write $\varpi_{s}$ as 
$$\varpi_{s}=\omega\otimes a\otimes (e_{r})_{k}\otimes b\otimes \hat{w}_{s}-\omega\otimes a\otimes (e_{r+p})_{k}\otimes b\otimes \hat{w}_{s}.$$

We have
\begin{eqnarray*}
&&\left((\frac{p-q}{2}-\frac{\nu_{r}}{2}-\frac{\mu (p+q)}{2})\otimes(E_{r,p+r})_{k}
+(\frac{q-p}{2}-\frac{\nu_{r}}{2}+\frac{\mu (p+q)}{2})\otimes(E_{p+r,r})_{k}\right)(\varpi_{s})\\
&=&-(\frac{p-q}{2}-\frac{\nu_{r}}{2}-\frac{\mu (p+q)}{2})\omega\otimes a\otimes(e_{r})\otimes b\otimes \hat{w}_{s}\\
&&\quad
+(\frac{q-p}{2}-\frac{\nu_{r}}{2}+\frac{\mu (p+q)}{2})\omega\otimes a\otimes (e_{r+p})\otimes b\otimes \hat{w}_{s},
\end{eqnarray*} and
$$\frac{p-q-\mu(p+q)}{2}\gamma_{k}(\varpi_{s})=\frac{p-q-\mu(p+q)}{2}\omega\otimes(a\otimes(e_{r})\otimes b\otimes \hat{w}_{s}+a\otimes(e_{p+r})\otimes b\otimes \hat{w}_{s}).$$

Then $$y_{k}(\varpi_{s})=(\frac{\nu_{r}}{2}+\frac{m_{r}+m_{r-1}}{2}-k+\frac{1}{2})(\varpi_{s}), \text{ for } m_{r-1}<k\leq m_{r}, r\leq p.$$

\subsubsection{Case 2:~$m_{r-1}<k\leq m_{r}$, for $r\leq p$ and $q=p$.} 
By a similar discussion, we have,
\begin{eqnarray*}
&&\tilde{y}_{k}(\varpi_{s})\\
&=&\left((-\frac{\nu_{r}}{2}-\mu p)\otimes(E_{r,p+r})_{k}
+(-\frac{\nu_{r}}{2}+\mu p)\otimes(E_{p+r,r})_{k}\right.\\
&&\quad 
+\mathop{\sum_{l\neq k}}_{m_{r-1} < l\leq m_{r}}\left(\frac{\gamma_{l}}{2}\otimes(E_{r,p+r})_{k}-\frac{\gamma_{l}}{2}\otimes(E_{p+r,r})_{k}\right)\\
&&\quad
+\sum_{r<t\leq p}\left(\sum_{m_{t-1}<l\leq m_{t}}(E_{p+r,p+t})_{l}\otimes (E_{p+t,r})_{k}+(E_{r,t})_{l}\otimes (E_{t,p+r})_{k}
\right)\\
&&\quad
\left.-\sum_{t<r}\left(\sum_{m_{t-1}<l\leq m_{t}}
(E_{p+r,p+t})_{l}\otimes (E_{t,p+r})_{k}+(E_{r,t})_{l}\otimes (E_{p+t,r})_{k}\right)\right)(\varpi_{s}).
\end{eqnarray*}

Thus by a similar discussion as in the $q>p$ case, we have
\begin{eqnarray*}
y_{k}(\varpi_{s})&=&\left((-\frac{\nu_{r}}{2}-\mu p)\otimes(E_{r,p+r})_{k}
+(-\frac{\nu_{r}}{2}+\mu p)\otimes(E_{p+r,r})_{k}\right.\\
&&\left.\quad-\mu p\gamma_{k}+\frac{m_{r}+m_{r-1}+1}{2}-k\right)(\varpi_{s})\\
&=&(\frac{\nu_{r}}{2}+\frac{m_{r}+m_{r-1}}{2}-k+\frac{1}{2})(\varpi_{s}), \text{ for } m_{r-1}<k\leq m_{r}, r\leq p, 
\end{eqnarray*}
which proved the theorem \ref{common-eigenvector} for the $q=p$ case.

\subsubsection{Case 3:~$k>m_{p}$.}

Now let $k>m_{p}$ and $q>p$. In this case, as an operator on $\varpi_{s}$, $\tilde{y}_{k}$ has the following simple form:
\begin{eqnarray*}
\tilde{y}_{k}
&=& -\sum_{k\neq l}
\sum_{1\leq i\leq p<j}(E_{p+j,p+i})_{l}\otimes (E_{i,p+j})_{k}\\
&=&-\sum_{r=1}^{p}\sum_{m_{r-1}<l\leq m_{r}}\sum_{j>p}
(E_{p+j,p+r})_{l}\otimes (E_{r,p+j})_{k}\\
&=&\sum_{l\leq m_{p}}
\frac{1}{2}(S_{kl}-S_{kl}\gamma_{k}\gamma_{l}).
\end{eqnarray*}

Then 
\begin{eqnarray*}
y_{k}(\varpi_{s})&=&\left(\tilde{y}_{k}+\frac{p-q-\mu (p+q)}{2}\gamma_{k}+\frac{1}{2}\sum_{l>k}S_{lk}-\frac{1}{2}\sum_{l<k}S_{lk}+\frac{1}{2}\sum_{l\neq k}S_{lk}\gamma_{l}\gamma_{k}\right)(\varpi_{s})\\
&=&\left(\sum_{l\leq m_{p}}
\frac{1}{2}(S_{kl}-S_{kl}\gamma_{k}\gamma_{l})-\frac{p-q-\mu (p+q)}{2}+\frac{1}{2}\sum_{l>k}S_{lk}-\frac{1}{2}\sum_{m_{p}<l<k}S_{lk}\right.\\
&&\quad\left. -\frac{1}{2}\sum_{l\leq m_{p}}S_{lk}
+\frac{1}{2}\sum_{l>k}S_{lk}+\frac{1}{2}\sum_{m_{p}<l<k}S_{lk}+\frac{1}{2}\sum_{l\leq m_{p}}S_{lk}\gamma_{l}\gamma_{k}\right)(\varpi_{s})\\
&=&\left(-\frac{p-q-\mu(p+q)}{2}+\sum_{l>k}S_{lk}\right)(\varpi_{s}).
\end{eqnarray*}

Notice the action of $\sum_{l>k}S_{lk}$ on $\varpi_{s}$ only affects  the component $\hat{w}_{s}$.
From the construction of $\hat{w}_{s}$, it is easy to see that 
$\sum_{l>k}S_{lk}(\hat{w}_{s})=\hat{\alpha}_{k-m_{p}+1,s}(\hat{w}_{s})$, where 
$\hat{\alpha}_{i,s}$ is defined in lemma \ref{eigenvector} and only depends on the partition $\xi^{\mu}$.

Thus in this case, we have 
$$y_{k}(\varpi_{s})=(-\frac{p-q-\mu(p+q)}{2}+\hat{\alpha}_{k-m_{p}+1,s})\varpi_{s}.$$

\subsection{Image of the Harish-Chandra modules}
We continue to use the setup in section \ref{sect-comm}.
Let $\mathcal{H}_{n}(\kappa_{1},\kappa_{2})$ be the
type $BC_{n}$ dAHA with parameters as in theorem \ref{aff}.

For $i=1,\ldots,p$, let $\mathcal{H}^{i}:=\mathcal{H}_{n_{i}^{\mu}}(\kappa_{1})$  be a type
$A_{n_{i}^{\mu}-1}$ dAHA generated by 
$\mathcal{S}_{n_{i}^{\mu}}$ and 
$y_{m_{i-1}+1}, \ldots, y_{m_{i}}$. 
Let $\mathcal{H}^{\xi}:=\mathcal{H}_{n_{\xi}^{\mu}}(\kappa_{1},\kappa_{2})$ be a type $BC_{n_{\xi}^{\mu}}$ dAHA generated by
$\mathcal{S}_{n_{\xi}^{\mu}}\ltimes \mathbb{Z}_{2}^{n_{\xi}^{\mu}}$ and 
$y_{m_{p}+1}, \ldots, y_{n}$. 
Then $\mathcal{H}:=\otimes_{i=1}^{p} \mathcal{H}^{i}\otimes \mathcal{H}^{\xi}$
is a subalgebra of $\mathcal{H}_{n}(\kappa_{1},\kappa_{2})$. 

Let $\mathcal{P}$ be the Specht module of $S^{n_{\xi}^{\mu}}$. 
Define a $\mathcal{H}$-module structure on $\mathcal{P}$ by
letting $\mathcal{S}_{n_{i}^{\mu}}$ act trivially,  
$\mathbb{Z}_{2}^{n_{\xi}^{\mu}}$ act 
by $-1$ and $y_{i}$ act on $\mathcal{P}$ by 
$y_{i}(\varpi_{s})=\lambda_{i,s}\varpi_{s}$ for $s=1,\ldots, d_{\xi^{\mu}}$. Here,
$\varpi_{s}$ and $\lambda_{i,s}$ are defined in theorem \ref{common-eigenvector}.

Define $$\tilde{\mathcal{P}}=\Ind_{\mathcal{H}}^{\mathcal{H}_{n}(\kappa_{1},\kappa_{2})}\mathcal{P}.$$
We have
\begin{theorem}{\label{main}}
The image of the Harish-Chandra module $H_{\pi\otimes\nu}$ under the functor $F_{n,p,\mu}$ is isomorphic to $\tilde{\mathcal{P}}$ as $\mathcal{H}_{n}(\kappa_{1},\kappa_{2})$-modules.
\end{theorem}

\begin{proof}
Define a linear map
\begin{eqnarray*}
\theta : &\mathcal{P}&\to F_{n,p,\mu}(H_{\pi\otimes\nu}),\\
&\tilde{\varpi}_{s}&\mapsto \varpi_{s},
\end{eqnarray*}
and extend it to $\tilde{\mathcal{P}}$ as a $\mathcal{H}_{n}(\kappa_{1},\kappa_{2})$-module homomorphism. It is surjective from proposition \ref{prop}.

Now compute the dimension of $\tilde{\mathcal{P}}$.  Since 
$$|\tilde{\Gamma}|=2^{n_{\xi}^{\mu}}n_{\xi}^{\mu}!\prod_{i=1}^{p}n^{\mu}_{i}!,$$
we have 
$$|\mathcal{W}_{BC_{n}}/\tilde{\Gamma}|=\frac{2^{n}n!}{2^{n_{\xi}^{\mu}}n^{\mu}_{\xi}!\prod_{i=1}^{p}n^{\mu}_{i}}=\frac{n!\prod_{i=1}^{p}2^{n_{p}^{\mu}}}{n^{\mu}_{\xi}!\prod_{i=1}^{p}n^{\mu}_{i}}.$$

Notice that $\dim \mathcal{P}=\dfrac{n_{\xi}^{\mu}!}{\Pi h_{k}(\xi^{\mu})}$, thus 
$$\dim \tilde{\mathcal{P}}=\frac{n!\prod_{i=1}^{p}2^{n_{p}^{\mu}}}{\prod_{i=1}^{p}n^{\mu}_{i}\Pi h_{k}(\xi^{\mu})}.$$

By comparing the dimension, we can see that $\theta$ is an 
isomorphism which proves the theorem.
\end{proof}

\begin{corollary}
When $n=1$, $F_{n,p,\mu}(H_{\pi\otimes\nu})$ is the principal series module for $\mathcal{H}_{1}(\kappa_{1},\kappa_{2})$ with character $\lambda_{1,1}$.
\end{corollary}

\section{Infinitesimal characters}
From now on, we assume the rank of the dAHA $n=1$ and $q>p$. Then  the possible $M$-module $W$ with the form \eqref{W} has the following parameters:
\begin{enumerate}
\item[Case 1:] All $n_{i}=\mu(q-p)+2\tau$ and $\gamma_{j}=-\mu p+\tau$ except one $n_{k}=\mu(q-p)+2\tau-1$ for some $1\leq k\leq p$; 
\item[Case 2:] all $n_{i}=\mu(q-p)+2\tau$ and $\gamma_{j}=-\mu p+\tau$ except $\gamma_{q-p}=-\mu p+\tau-1$.
\end{enumerate}

We want to compute $y^{2}_{1}$ by using  the infinitesimal character of $\mathfrak{g}=\mathfrak{gl}_{N}$. Most of the computations in this section are done by using the software Mathematica 6.0.

\subsection{Casimir elements and $y_{1}^{2}$}
The Casimir elements are a family of elements in $Z(\mathfrak{g})\subset U(\mathfrak{g})$ which are defined by:
\begin{equation*}
C_{k}=\sum_{i_{1},\ldots,i_{k}}E_{i_{1},i_{2}}E_{i_{2},i_{3}}\cdots E_{i_{k},i_{1}}, \text{ for } k=1,\ldots, n. 
\end{equation*}
Define the action of $C_{k}$ on $F_{n,p,\mu}(H_{\pi\otimes \nu})$ by 
letting it act on the $0$-th component, i.e., on the $H_{\pi\otimes \nu}$ part.

\begin{proposition}{\label{ycc}}
Suppose $n=1$. As operators on $F_{n,p,\mu}(H_{\pi\otimes\nu})$, \begin{eqnarray}{\label{eqn-ycc}}
y_{1}^{2}&=&-\frac{1}{3}C_{3}+\frac{1}{2}(\frac{p+q}{3}+\mu(q-p)+2\tau)C_{2}+\frac{(p+q)^{2}}{12}-\frac{1}{3}\\\nonumber
&&\qquad+\frac{1}{4}(p-q-2\tau)^{2}\mu^{2}-\frac{1}{6}pq(p-q)(p+q)\mu(1-\mu^{2})\\\nonumber
&&\qquad+\frac{1}{6}(p+q)\tau(2-(p+q)\tau+3\mu\tau(p-q)-4\tau^{2}).
\end{eqnarray}
\end{proposition}
\begin{proof}
The proof is based on direct computations.
At first, as operators on $F_{n,p,\mu}(H_{\pi\otimes\nu})$, we have 
\begin{eqnarray*}
&&\tilde{y}_{1}^{2}\\
&=&\sum_{i|kj}E_{i,j}E_{k,i}\otimes E_{j,k}\\
&=&-\sum_{i|kj}E_{i,j}E_{k,i}E_{j,k}\otimes 1-(\mu p-\tau)\sum_{i\leq p<j}E_{i,j}E_{j,i}\otimes 1+(\mu q+\tau)\sum_{j\leq p<i}E_{i,j}E_{j,i}\otimes 1\\
&=&-\sum_{i|k,j}E_{i,j}(E_{j,k}E_{k,i}+E_{k,k}\delta_{i,j}-E_{j,i})\otimes 1-(\mu p-\tau)\sum_{i\leq p<j}E_{i,j}E_{j,i}\otimes 1
\\&&\qquad+(\mu q+\tau)\sum_{j\leq p<i}E_{i,j}E_{j,i}\otimes 1\\
&=&-\sum_{i|kj}E_{i,j}E_{j,k}E_{k,i}\otimes 1+(p+q+\mu(q-p)+2\tau)\sum_{i\leq p<j}E_{i,j}E_{j,i}\otimes 1\\
&&\quad-p(p+\mu q+\tau)\otimes\frac{1-\gamma}{2}+(p+\mu q+\tau)q\otimes\frac{1+\gamma}{2}-\mu pq(p+q)(p+\mu q+\tau),
\end{eqnarray*}
and 
\begin{eqnarray*}
C_{2}&=&\sum_{i,j}E_{i,j}E_{j,i}\\
&=&\sum_{i\leq p<j}E_{i,j}E_{j,i}+\sum_{j\leq p<i}E_{i,j}E_{j,i}+\sum_{i,j>p}E_{i,j}E_{j,i}+\sum_{i,j\leq p}E_{i,j}E_{j,i}\\
&=&\sum_{i\leq p<j}E_{i,j}E_{j,i}+\sum_{j\leq p<i}(E_{j,i}E_{i,j}+E_{i,i}-E_{j,j})+\sum_{i,j>p}E_{i,j}E_{j,i}+\sum_{i,j\leq p}E_{i,j}E_{j,i}\\
&=&2\sum_{i\leq p<j}E_{i,j}E_{j,i}+(q-p+2\mu p-2\tau)\otimes\frac{1-\gamma}{2}+(p+q-2\mu q-2\tau)\otimes \frac{1+\gamma}{2}\\
&&\qquad+q(\mu p-\tau)^{2}+p(\mu q+\tau)^{2}-\mu pq(p+q).
\end{eqnarray*}

For $C_{3}$, we have:
\begin{eqnarray*}
C_{3}
&=&\sum_{i,j,k}E_{i,j}E_{j,k}E_{k,i}\\
&=&3\sum_{i|jk}E_{i,j}E_{j,k}E_{k,i}+\sum_{ijk}E_{i,j}E_{j,k}E_{k,i}-2\sum_{i|jk}E_{i,j}E_{j,i}+2\sum_{i|jk}E_{k,j}E_{j,k}\\
&&\qquad-\sum_{i|j}E_{i,i}E_{j,j}-\sum_{i|j}E_{j,j}E_{i,i}.
\end{eqnarray*}
As operators on $F_{n,p,\mu}(H_{\pi\otimes\nu})$,
\begin{eqnarray*}
&&\sum_{ijk}E_{i,j}E_{j,k}E_{k,i}\\
&=&(-3(\mu q+\tau)^{2}+3p(\mu q+\tau))\frac{1+\gamma}{2}
+(-3(\mu p-\tau)^{2}-3q(\mu p-\tau))\frac{1-\gamma}{2}\\
&&\quad
+p(\mu q+\tau)^{3}-q(\mu p-\tau)^{3}-1,
\end{eqnarray*}
\begin{eqnarray*}
&&-2\sum_{i|jk}E_{i,j}E_{j,i}+2\sum_{i|jk}E_{k,j}E_{j,k}\\
&=&-2(p+q)\sum_{i\leq p<j}E_{i,j}E_{j,i}
+(2p^{2}+2pq+4p(\mu p-\tau))\frac{1-\gamma}{2}
\\&&\quad-4q(\mu q+\tau)\frac{1+\gamma}{2}
+2pq((\mu p-\tau)^{2}+(\mu q+\tau)^{2})+2\mu pq(p^{2}+pq),
\end{eqnarray*}
and 
\begin{eqnarray*}
&&-\sum_{i|j}E_{i,i}E_{j,j}-\sum_{i|j}E_{j,j}E_{i,i}\\
&=&2p(\mu q+\tau)(1-\gamma)-2q(\mu p-\tau)(1+\gamma)
+4pq(\mu p-\tau)(\mu q+\tau).
\end{eqnarray*}

Thus we have
\begin{eqnarray*}
&&C_{3}\\
&=&3\sum_{i|jk}E_{i,j}E_{j,k}E_{k,i}-1-2(p+q)\sum_{i\leq p<j}E_{i,j}E_{j,i}\\
&&\quad+(-3(\mu p-\tau)^{2}-3q(\mu p-\tau)+2p^{2}+2pq+4\mu p(p+q))\frac{1-\gamma}{2}\\
&&\quad+(-3(\mu q+\tau)^{2}+3p(\mu q+\tau)-4\mu q(p+q))
\frac{1+\gamma}{2}+2pq((\mu p-\tau)^{2}+(\mu q+\tau)^{2})\\
&&\quad 
+2\mu pq(p^{2}+pq)+p(\mu q+\tau)^{3}
-q(\mu p-\tau)^{3}+4pq(\mu p-\tau)(\mu q+\tau).
\end{eqnarray*}

Notice that
$$y_{1}^{2}=(\tilde{y}_{1}+\frac{p-q-\mu(p+q)}{2}\gamma)^{2}=\tilde{y}_{1}^{2}+\frac{(p-q-\mu(p+q))^{2}}{4}.$$
So we have \eqref{eqn-ycc}.

\end{proof}

\begin{remark}\label{yccpequq}
Proposition \ref{ycc} also holds when $p=q$. In that case, we have a simpler 
formula:
\begin{eqnarray*}
y_{1}^{2}&=&-\frac{1}{3}C_{3}+(\frac{p}{3}+\tau)C_{2}+\frac{p^{2}}{3}-\frac{1}{3}+\tau^{2}+\frac{2}{3}p\tau(1-p\tau-2\tau^{2}).
\end{eqnarray*}
\end{remark}

\subsection{Infinitesimal characters for principal series modules}
Let $H_{\pi\otimes \nu}$ be the principal series module associated to the 
irreducible representation $(\pi, W)$ of $M$ and $\nu\in\mathfrak{a}^{*}_{\mathbb{C}}$. Let $\mathfrak{m}$ be the complexification of the Lie algebra of $M$, and $\mathfrak{t}^{s}$ be its Cartan subalgebra.
Let $\Delta^{+}(\mathfrak{m},\mathfrak{t}^{s})$ be the set of positive roots with respect to $(\mathfrak{m},\mathfrak{t}^{s})$.
Let $$\rho(\mathfrak{m})=\frac{1}{2}\sum_{\alpha\in \Delta^{+}(\mathfrak{m},\mathfrak{t}^{s})}\alpha.$$
Suppose $(\pi, W)$ has a highest weight vector with highest weight
$\lambda_{0}\in (\mathfrak{t}^{s})^{*}$. 

\begin{lemma}[See \cite{Vog}]
The infinitesimal character of $H_{\pi\otimes\nu}$ is 
$(\lambda,\nu)$, where
$\lambda=\lambda_{0}+\rho(\mathfrak{m})$. 
Moreover, for any $z\in Z(\mathfrak{g})$, $z$ acts on $H_{\pi\otimes\nu}$ by
$$(\lambda, \nu)(\xi(z)),$$
where $\xi$ is the Harish-Chandra map for $G$.
\end{lemma}

Choose the Cartan subalgebra $\mathfrak{t}$ of $\mathfrak{g}$ to be the subalgebra consisting of all diagonal matrices. Let $\varphi: Z(\mathfrak{g})\to U(\mathfrak{t})$ be the restriction map.
Then we have
\begin{eqnarray*}
\varphi(C_{2})&=&\sum_{i=1}^{p+q}E_{ii}^{2}+\sum_{i=1}^{p+q}(p+q-2i+1)E_{ii};\\
\varphi(C_{3})&=&\sum_{i=1}^{p+q}E_{ii}^{3}+\sum_{i=1}^{p+q}(2(p+q)-3i+1)E_{ii}^{2}+\sum_{i=1}^{p+q}(p+q-i)(p+q-2i+1)E_{ii}\\
&&\qquad+\sum_{i<j}(-p-q+2i-1)E_{jj}-\sum_{i<j}E_{ii}E_{jj}.
\end{eqnarray*}
 
Let 
$$B=\frac{1}{\sqrt{2}}\left(\begin{array}{ccc}-I_p & I_p &  \\I_p & I_p &  \\ &  &\sqrt{2}I_{q-p}\end{array}\right).$$
Then we have 
$$B^{2}=I_{p+q},\text{ and} \qquad B\left(\begin{array}{ccc}h & a &  \\a & h &  \\ &  & I_{q-p}\end{array}\right)B=\left(\begin{array}{ccc}-a+h &  &  \\ & a+h &  \\ &  & b\end{array}\right),$$
where $a=\diag(a_{1},\ldots,a_{p}), h=\diag(h_{1},\ldots,h_{p}), b=\diag(b_{1},\ldots, b_{q-p})$. 

Then $B$ conjugates the Cartan subalgebra $\mathfrak{t}$ of $\mathfrak{g}$
to another Cartan subalgebra $\mathfrak{a}+\mathfrak{t}^{s}$.
Since the Harish-Chandra map $\xi$ corresponding to the Cartan subalgebra $\mathfrak{t}$ is 
$$\xi: E_{ii}\mapsto E_{ii}-\frac{p+q-2i+1}{2},$$
then the action of $C_{i}$ on $H_{\pi\otimes\nu}$ is the multiplication by
$$c_{i}=(\lambda, \nu)(\xi(B\varphi(C_{i})B)).$$

Now consider the special case.
\vspace{.5cm}
\paragraph{\bf Case 1}

\begin{eqnarray*}
\lambda&=&(\mu(q-p)+2\tau,\ldots,\mu(q-p)+2\tau,\mu(q-p)+2\tau-1,\mu(q-p)+2\tau,\ldots,\\&&\qquad\mu(q-p)+2\tau,\frac{q-p-1}{2}-\mu p+\tau,\ldots, -\frac{q-p-1}{2}-\mu p+\tau),\\
\nu&=&(\nu_{1},\ldots, \nu_{p}),\text{ where $\mu(q-p)+2\tau-1$ appears in the $k$-th position}.
\end{eqnarray*}
We can find that 
\begin{eqnarray*}
c_{2} &=& \frac{1}{2}\sum_{i=1}^{p}\nu_{i}^{2}-\frac{p^{3}}{6}-\frac{pq^{2}}{2}+\frac{p}{6}+\frac{1}{2}+\mu (p-q)+\frac{-p^{3}\mu^{2}+pq^{2}\mu^{2}}{2}-2\tau+\tau^{2}(p+q);\\
c_{3} &=& \frac{p+q}{4}\sum_{i=1}^{p}\nu_{i}^{2}-\frac{3\nu_{k}^{2}}{4}-\frac{pq^{3}}{4}-\frac{p^{2}q^{2}}{4}+\frac{q^{2}}{4}-\frac{p^{3}q}{12}+\frac{7pq}{12}+\frac{q}{4}-\frac{p^{4}}{12}+\frac{p^{2}}{3}+\frac{p}{4}-1\\
&&\quad -\frac{3\mu(p-q)}{4}\sum_{i=1}^{p}\nu_{i}^{2}+\frac{1}{4}\mu^{3}pq^{3}-\frac{3}{4}\mu^{3}p^{2}q^{2}-\frac{1}{4}\mu^{3}p^{3}q+\frac{3}{4}\mu^{3}p^{4}+\frac{1}{4}\mu^{2}pq^{3}\\
&&\quad+\frac{1}{4}\mu^{2}p^{2}q^{2}-\frac{3}{4}\mu^{2}q^{2}-\frac{1}{4}\mu^{2}p^{3}q+\frac{3}{2}\mu^{2}pq-\frac{1}{4}\mu^{2}p^{4}-\frac{3}{4}\mu^{2}p^{2}-\frac{1}{4}\mu pq^{3}+\frac{3}{4}\mu p^{2}q^{2}\\
&&\quad-\frac{1}{2}\mu q^{2}-\frac{3}{4}\mu p^{3}q+\frac{1}{4}\mu pq+\frac{3}{4}\mu q+\frac{1}{4}\mu p^{4}+\frac{1}{4}\mu p^{2}-\frac{3}{4}\mu p+\frac{3\tau}{2}\sum_{i=1}^{p}\nu_{i}^{2}+\tau^{3}(p+q)\\
&&\quad-3\tau^{2}+\frac{3}{2}\mu^{2}\tau pq^{2}-\frac{3}{2}\mu^{2}\tau p^{3}-3\mu\tau q+3\mu\tau p-\frac{3}{2}\tau pq^{2}-\frac{1}{2}\tau p^{3}+\frac{1}{2}\tau p+\frac{3}{2}\tau.
\end{eqnarray*}

Then from proposition \ref{ycc}, we have
\begin{eqnarray*}
y_{1}^{2}
=\frac{\nu_{k}^{2}}{4}=\lambda_{1,1}^{2},
\end{eqnarray*}
which is compatible with the eigenvalue we found in theorem \ref{common-eigenvector}.

\begin{remark}
When $q=p$, from remark \ref{yccpequq}, we know that it is similar
to case 1. In that case, we have
\begin{eqnarray*}
\lambda&=&(2\tau,\ldots,2\tau,2\tau-1,2\tau,\ldots,2\tau),\\
\nu&=&(\nu_{1},\ldots, \nu_{p}), \text{ where $-1$ appears in the $k$-th position.}
\end{eqnarray*}
Then $c_{2}$ and $c_{3}$ can be obtained by letting $p=q$ in the above formulas. We get similar results which are compatible to theorem \ref{common-eigenvector}.
\end{remark}

\vspace{.5cm}
\paragraph{\bf Case 2}
\begin{eqnarray*}
\lambda&=&(\mu(q-p)+2\tau,\ldots,\mu(q-p)+2\tau,\frac{q-p-1}{2}-\mu p+\tau,\ldots,-\frac{q-p-1}{2}-\mu p+\tau-1),\\
\nu&=&(\nu_{1},\ldots, \nu_{p}).
\end{eqnarray*}

In this case, we can find that 
\begin{eqnarray*}
c_{2} &=& \frac{1}{2}\sum_{i=1}^{p}\nu_{i}^{2}-\frac{pq^{2}}{2}+q-\frac{p^{3}}{6}-\frac{5p}{6}+\frac{1}{2}\mu^{2}pq^{2}-\frac{1}{2}\mu^{2}p^{3}+2\mu p-2\tau+\tau^{2}(p+q);\\
c_{3} 
&=& \frac{p+q}{4}\sum_{i=1}^{p}\nu_{i}^{2}-\frac{pq^{3}}{4}-\frac{p^{2}q^{2}}{4}-\frac{p^{3}q}{12}+\frac{25pq}{12}-\frac{p^{4}}{12}-\frac{11p^{2}}{12}-1\\
&&\quad +\frac{3\mu(q-p)}{4}\sum_{i=1}^{p}\nu_{i}^{2}+\frac{1}{4}\mu^{3}pq^{3}-\frac{3}{4}\mu^{3}p^{2}q^{2}-\frac{1}{4}\mu^{3}p^{3}q+\frac{3}{4}\mu^{3}p^{4}+\frac{1}{4}\mu^{2}pq^{3}\\
&&\quad+\frac{1}{4}\mu^{2}p^{2}q^{2}-\frac{1}{4}\mu^{2}p^{3}q-\frac{1}{4}\mu^{2}p^{4}-3\mu^{2}p^{2}-\frac{1}{4}\mu pq^{3}+\frac{3}{4}\mu p^{2}q^{2}-\frac{3}{4}\mu p^{3}q\\
&&\quad-\frac{7}{4}\mu pq+\frac{1}{4}\mu p^{4}+\frac{15}{4}\mu p^{2}-\frac{5}{2}\tau p-\frac{1}{2}\tau p^{3}+3\tau q-\frac{3}{2}\tau pq^{2}+6\mu\tau p\\
&&\quad-\frac{3}{2}\mu^{2}\tau p^{3}+\frac{3}{2}\mu^{2}\tau pq^{2}-3\tau^{2}+\tau^{3}(p+q).
\end{eqnarray*}

Then from proposition \ref{ycc}, we have that $y_{1}^{2}$ acts by
\begin{eqnarray*}
y_{1}^{2}
=(\frac{q-p+\mu(p+q)}{2})^{2}=\lambda_{1,1}^{2},
\end{eqnarray*}
which is also compatible with the eigenvalue we found in theorem \ref{common-eigenvector}.

\section*{acknowledgments}
{The author thanks Pavel Etingof for giving this problem and many useful discussions. The author thanks David Vogan for many explainations about the principal series representations. The author also thanks Ju-Lee Kim, Jun Yu and Ting Xue for many useful discussions. His work was partially supported by the NSF grant DMS-0504847.}


\end{document}